\documentclass{bolmat5}
\usepackage{amsfonts}
\usepackage{amssymb}
\usepackage{amsmath}
\usepackage{tikz-cd}
\usepackage{float}
\usepackage{tikz}
\usetikzlibrary{trees}
\usetikzlibrary{decorations.pathmorphing}
\usetikzlibrary{decorations.pathreplacing}
\usetikzlibrary{decorations.markings}
\usetikzlibrary{decorations.shapes}
\usepackage{verbatim}
\usepackage{amsmath}
\usepackage{amsthm}
\usepackage{dsfont}
\usepackage[T1]{fontenc}
\usepackage{babel}[spanish]
\usepackage{makeidx}
\usepackage{enumerate}
\usepackage{calc}
\usepackage{latexsym}
\usepackage{amssymb}
\usepackage{amsfonts}
\usepackage[mathscr]{eucal}
\usepackage[shortlabels]{enumitem}
\usepackage{amscd}
\usepackage{array}
\usepackage{amsrefs}
\usepackage{tikz}
\usepackage{listings}
\usepackage{chngcntr}
\usepackage{setspace}
\usepackage{xcolor}
\usepackage{bigints}

\newcommand{\xitem}{
  \par\hangindent3em\hangafter1
  \noindent\makebox[3em][l]{$\triangleright$}%
  \ignorespaces}
 
 \newcommand{\enter}[0]{
 \newline\newline
 }
 
\newcommand{\vectordos}[2]{
\begin{pmatrix}
#1\\#2
\end{pmatrix}
}
\newcommand{\vectortres}[3]{
\begin{pmatrix}
#1\\#2\\#3
\end{pmatrix}
}

\labeldocument[firstpage = 1, volume = 0, number = 0, month = 00, year = 1900, day = 00, monthreceived = 0, yearreceived = 1900, monthaccepted = 0, yearaccepted = 1900]

\begin{document}

\setlength{\parindent}{0pt}

\title[maintitle = {Numerical approximation of partial differential equations with MFEM library},
	othertitle = {Undergraduate Degree Work - Mathematics\\Directed by: Boyan Lazarov \& Juan Galvis},
	shorttitle = {}
]
\begin{authors}[] 
\author[firstname = {Felipe},
	surname = {Cruz},
	institutionnumber = {1},
	email = {fcruzv@unal.edu.co},
]
\end{authors}
\begin{affiliations}
\affiliation[
	department = {Departamento de Matem\'aticas},
	institution = {Universidad Nacional de Colombia},
	city = {Bogot\'a D.C.},
	country = {Colombia}
]
\end{affiliations}

\begin{mainabstract}
We revise the finite element formulation for Lagrange, Raviart-Thomas, and Taylor-Hood finite element spaces. We solve Laplace equation in first and second order formulation, and compare the solutions obtained with Lagrange and Raviart-Thomas finite element spaces by changing the order of the shape functions and the refinement level of the mesh. Finally, we solve Navier-Stokes equations in a two dimensional domain, where the solution is a steady state, and in a three dimensional domain, where the system presents a turbulent behaviour. All numerical experiments are computed using MFEM library, which is also studied.\enter
\keywords{Finite Elements, MFEM library, Lagrange, Raviart-Thomas, Taylor-Hood, Laplace Equation, Navier-Stokes Equations.}
\end{mainabstract}
	
\begin{otherabstract}
Revisamos la formulaci\'on de elementos finitos para los espacios de elementos finitos de Lagrange, Raviart-Thomas y Taylor-Hood. Solucionamos la ecuaci\'on de Laplace en su formulaci\'on de primer y segundo orden, y comparamos las soluciones obtenidas con los espacios de elementos finitos de Lagrange y Raviart-Thomas al cambiar el orden de las funciones base y el nivel de refinamiento de la malla. Finalmente, resolvemos las ecuaciones de Navier-Stokes en un dominio bidimensional, donde la soluci\'on es un estado estable, y en un dominio tridimensional, donde el sistema presenta un comportamiento turbulento. Todos los experimentos num\'ericos se realizan utilizando la librer\'ia MFEM, la cual es tambi\'en estudiada.\enter
\keywords{Elementos Finitos, Librer\'ia MFEM, Lagrange, Raviart-Thomas, Taylor-Hood, Ecuaci\'on de Laplace, Ecuaciones de Navier-Stokes.}
\end{otherabstract}
\newpage
\tableofcontents
\newpage
\section{Preliminaries}\label{sec:theory}
In this section we are going to recall the theoretical background needed in the rest of the paper. First, we are going to review the finite element methods used for the Laplace equation in second and first order form.  We write the strong and weak form of the problem and introduce the Lagrange and mixed finite element spaces. We also introduce the 
MFEM library by giving an overview of its main characteristics and the general structure of a finite element code in MFEM.

\subsection{Partial Differential Equations}\label{sec:pde}
For the scope of this work, partial differential equations are of the form $$F\left(u,p,\frac{\partial p}{\partial x_i},\frac{\partial u}{\partial x_i},\frac{\partial^2 p}{\partial x_i^2},\frac{\partial^2 u}{\partial x_i^2},f\right)=0;\ \ i=1,2,3,4$$ where $u:\Omega\subseteq\mathbb{R}^3\rightarrow U\subseteq\mathbb{R}^3$, $p:\Omega\subseteq\mathbb{R}^3\rightarrow P\subseteq\mathbb{R}$, $f$ is a restriction parameter and $F$ is \textbf{any} mathematical expression within its variables.\enter
It is well known that PDEs have multiple solutions, in fact, there is a vectorial space consisting of all the solutions for a given equation. For this reason, the equation is usually presented with a boundary condition that forces a unique solution for the equation \cite{Lagrange}. If some problem is being solved in a given domain $\Omega\subseteq\mathbb{R}^3$, whose boundary is $\Gamma$, then the problem has the form
\begin{equation*}
    \left\{\begin{split}&F\left(u,p,\frac{\partial p}{\partial x_i},\frac{\partial u}{\partial x_i},\frac{\partial^2 p}{\partial x_i^2},\frac{\partial^2 u}{\partial x_i^2},f\right)=0\text{  in }\Omega,\\
    &F_0\left(u,p,\frac{\partial p}{\partial x_i},\frac{\partial u}{\partial x_i},\frac{\partial^2 p}{\partial x_i^2},\frac{\partial^2 u}{\partial x_i^2},g\right)=0\text{  in }\Gamma.\end{split}\right.
\end{equation*}
The two main equations that we treat are the Laplace equation and the Navier-Stokes equation, whose represent a fluids phenomenon in real life. Note that $x_1=x,x_2=y,x_3=z,x_4=t$ (in time-space interpretation) and that, $p$ is the fluid's pressure and $u$ is the fluid's velocity.
\subsubsection{Laplace equation}\label{sec:laplace}
Laplace equation consists on finding $p:\Omega\subseteq\mathbb{R}^3\rightarrow\mathbb{R}$ such that 
\begin{equation}\label{eq:laplace}
    -\Delta p=f,\text{  in }\Omega
\end{equation}
where $f:\Omega\rightarrow\mathbb{R}$ is a given function and $\Delta p = \nabla\cdot\nabla p = \frac{\partial^2 p}{\partial x^2}+\frac{\partial^2 p}{\partial y^2}+\frac{\partial^2 p}{\partial z^2}$ \cite{Galvis}. It is clear that the equation is a partial differential equation of the form presented before in Section \ref{sec:pde} because it only depends on the second order derivatives of $p$.\enter
Now, we present two common ways of imposing a boundary condition to the problem: Dirichlet and Neumann. Let $\Gamma$ be the boundary of $\Omega$.\enter
\texttt{Dirichlet boundary condition}\\
The Dirichlet boundary condition is 
\begin{equation}\label{eq:dirichlet}
    p=g\text{ in }\Gamma
\end{equation}
where $g:\Gamma\rightarrow\mathbb{R}$ is a given function \cite{Galvis}. If $g=0$, the condition is called homogeneous Dirichlet boundary condition.
\enter
\texttt{Neumann boundary conditions}\\
The Neumann boundary condition is
\begin{equation}\label{eq:neumann}
    -\nabla p\cdot\eta=h\text{ in }\Gamma
\end{equation}
where $h:\Gamma\rightarrow\mathbb{R}$ is a given function and $\eta$ is the boundary's normal vector \cite{Galvis}.
\enter
As seen on \cite{Galvis}, problem \eqref{eq:laplace} can be stated with the two types of boundary condition \eqref{eq:dirichlet} and \eqref{eq:neumann}, by imposing each condition on different parts of $\Gamma$. However, the version of the problem presented later on Section \ref{sec:Lagrange} has the homogeneous Dirichlet boundary condition.\enter
Finally, the Laplace equation \eqref{eq:laplace} can be formulated on two different ways: first order formulation and second order formulation. The second order formulation is the one presented already on \eqref{eq:laplace}, which involves second order derivatives. Now, if we set $u=\nabla p$, the problem can be stated \cite{Galvis} as 
\begin{equation}\label{eq:formulacaoordemuno}
    \left\{\begin{split}
    \mbox{div}(u)=-f\text{  in }\Omega,\\
    u=\nabla p\text{  in }\Omega,
\end{split}\right.
\end{equation}
which is the first order formulation for the problem and, notice that it only involves first order derivatives of $u$ and $p$.\enter
On Section \ref{sec:Vs}, we compare both formulations of Laplace equation \eqref{eq:laplace} and \eqref{eq:formulacaoordemuno}. We use Lagrange finite elements to solve \eqref{eq:laplace} and Raviart-Thomas finite elements to solve \eqref{eq:formulacaoordemuno}.
\newpage
\subsubsection{Navier-Stokes equations}\label{sec:navierstokes}
In this section we revise incompressible Stokes and Navier-Stokes equations with Dirichlet boundary condition. First, incompressible Navier-Stokes equations consist on finding the velocity $u:\Omega\rightarrow\mathbb{R}^3$ and the pressure $p\rightarrow\mathbb{R}$ that solve the system of equations \eqref{eq:NavierStokes} \cite{TH}.
\begin{equation}\label{eq:NavierStokes}
    \begin{split}
        \frac{\partial u}{\partial t}+(u\cdot\nabla)u-\nu\Delta u+\nabla p=f,\text{ in }\Omega,\\
        \nabla\cdot u = 0,\text{ in }\Omega,\\
        u=g,\text{ in } \Gamma,
    \end{split}
\end{equation}
where $\Omega\subseteq\mathbb{R}^3$ is the spatial domain, $\Gamma=\partial\Omega$ the boundary of the domain, $\nu$ is called the kinematic viscosity (more information below), $f$ is the forcing term (given function), and $g$ is the specified Dirichlet boundary condition.\enter
On the other hand, when removing the non-linear term $(u\cdot\nabla)u$ from the first equation, we get Stokes linear equations which is system of equations \eqref{eq:Stokes} \cite{TH}.
\begin{equation}\label{eq:Stokes}
    \begin{split}
        \frac{\partial u}{\partial t}-\nu\Delta u+\nabla p=f,\text{ in }\Omega,\\
        \nabla\cdot u = 0,\text{ in }\Omega,\\
        u=g,\text{ in } \Gamma.
    \end{split}
\end{equation}
Take into account that $u(x,y,z,t)=(u_x,u_y,u_z)$ is a vector for each point in time-space and that $p(x,y,z,t)=p$ is a scalar for each point in time-space. For real life fluid problems modeled by these equations, $(x,y,z)$ is the position of a fluid's particle in space and $t$ is the time. Also, the equation $\nabla\cdot u=0$ is the one that establishes the incompressible condition for the fluid.\enter
\texttt{Kinematic Viscosity}\\
On the modeling of fluid flow, as presented on \cite{NS}, a very important parameter appears and its called \textit{Reynolds number}, $Re$. On \cite{NS}, it is defined as $$Re=\frac{\rho LU}{\mu},$$ where $\rho$ is the density of the fluid, $L$ is the characteristic linear dimension of the domain of the flow, $U$ is some characteristic velocity, and $\mu$ is the fluid's viscosity. For example, as seen on \cite{NS}, $L$ is defined supposing that $\Omega=[0,L]^n$, i.e., $L$ is the length of the domain on each direction, and $U$ can be taken as the square root of the average initial kinetic energy in $\Omega$.\enter
However, for the purpose of this work we neglect the parameters $L$ and $U$ in such way that $Re=\frac{\rho}{\mu}$. Therefore, the kinematic viscosity of a fluid is, according to \cite{NS}, \begin{equation}\label{eq:reynolds}
    \nu:=\frac{\mu}{\rho}=\frac{1}{Re}.
\end{equation} 
Notice that if the viscosity of the fluid, $\mu$, is higher and the density of the fluid, $\rho$, is lower, then, the kinematic viscosity of the fluid, $\nu$, is higher. This parameter quantifies the resistance that a fluid imposes to movement due to an external force, like gravity \cite{NS}.\enter
On Section \ref{sec:NSExp}, we do some numerical experiments in 2D and 3D using the Navier Miniapp of MFEM library, which solves \eqref{eq:NavierStokes}, by using the formulation presented on Section \ref{sec:NSFEM}.

\subsection{The Finite Element Method}\label{FEMSummary}
First, on Section \ref{sec:elliptic} we show how the solution for a differential equation is also a solution for a minimization problem and a variational problem. This serves as a basic case for showing that partial differential equations are, in fact, solved via minimization or variational problems; which are the ones solved with finite element methods.\enter
Then, in Sections \ref{sec:Lagrange} and \ref{sec:Mixed} we study two finite element methods. On both of them, the following procedure was applied:
\begin{enumerate}
    \item Consider the problem of solving Laplace equation \eqref{Laplace} with homogeneous Dirichlet boundary condition. 
    \item Multiply by some function (\textit{test function}) and integrate by parts. Apply boundary conditions.
    \item Discretize the domain and select  finite-dimensional function spaces for the solution and the test functions.
    \item Produce a  matrix system to solve for solution weights in the linear combination representation of the approximated solution.
\end{enumerate}
The basis functions that form part of the finite-dimensional spaces are called \textit{shape functions}. In Lagrange formulation, those are the functions in $V_h$, and in mixed formulation, those are the functions in $H_h^k$ and $L_h^k$, where the parameter $h$ denotes the size of the elements in the triangulation of the domain. Moreover, in Lagrange formulation the boundary condition is essential, and in mixed formulation, it is natural.\enter
Furthermore, on Section \ref{sec:HighOrder} we show high order shape functions used in Lagrange finite element method and, on Section \ref{sec:NSFEM} we study the most common finite elements used to approximate the solution for Navier-Stokes equations.

\subsubsection{FEM for elliptic problems}\label{sec:elliptic}
Let $\mathbb{D}$ be the two-point boundary value problem \eqref{eq:D} taken from \cite{Lagrange}.
\begin{equation}\label{eq:D}
    \begin{split}
        &-u''(x)=f(x),\ x\in(0,1),\\
        &u(0)=u(1)=0,
    \end{split}
\end{equation}
where $f$ is a given continuous function and $u''(x)=\frac{d^2u(x)}{dx^2}$.
Notice that this problem is a partial differential equation of the form presented on section \ref{sec:pde} but with all the functions being of the form $\mathbb{R}\rightarrow\mathbb{R}$ and $\Omega=(0,1)$ being a 1D domain, with homogeneous boundary condition ($0$ on $\Gamma=\{0,1\}$).\enter
By integrating $-u''(x)=f(x)$ twice, it is clear that the problem \eqref{eq:D} has a unique solution $u$ \cite{Lagrange}. For example, if $f(x)=e^x$
\begin{equation*}
    \begin{split}
        &-u''(x)=e^x\\
        \implies&-u'(x)=e^x+c_1\\
        \implies&-u(x)=e^x+c_1x+c_2
    \end{split}
\end{equation*}
and by applying boundary conditions,
\begin{equation*}
    \left\{
    \begin{split}
        &u(0)=0\implies c_2=-1,\\
        &u(1)=0\implies c_1=1-e.
    \end{split}\right.
\end{equation*}
It follows that $u(x)=-e^x+(e-1)x+1$ is the unique solution. Recall that the boundary conditions force the problem to have a unique solution, as mentioned previously on the work.\enter
Now, following \cite{Lagrange}, define the linear space $V$ of all continuous functions on $[0,1]$ that vanish at $\{0,1\}$, whose derivative is piecewise continuous and bounded on $[0,1]$. Also, define the linear functional $F:V\rightarrow\mathbb{R}$, by
\begin{equation*}
    F(v)=\frac{1}{2}\int_0^1\left[v'(x)\right]^2dx-\int_0^1f(x)v(x)dx.
\end{equation*}
With this settled, let $\mathbb{M}$ be the optimization problem of finding $u\in V$ such that 
\begin{equation}\label{eq:opt}
    F(u)\leq F(v)
\end{equation}
for all $v\in V$, and let $\mathbb{V}$ be the variational problem of finding $u\in V$ such that
\begin{equation}\label{eq:var}
    \int_0^1u'(x)v'(x)dx=\int_0^1f(x)v(x)dx
\end{equation}
for all $v\in V$.\enter
Remark that if $u,v\in V$ then $w=u+\alpha v\in V$ for any $\alpha\in\mathbb{R}$, because $w$ is continuous in $[0,1]$; $w(1)=w(0)=u(0)+\alpha v(0)=0$, i.e. vanishes on $\{0,1\}$; and $w'=u'+\alpha v'$ is piecewise continuous and bounded on $[0,1]$. The rest of this section is focused on showing that $\mathbb{D}$, $\mathbb{M}$ and $\mathbb{V}$ are equivalent problems \cite{Lagrange}.\enter
\texttt{Equivalence }$\mathbb{D}\iff\mathbb{V}$\\
Let $u_\mathbb{D}$ be the solution for $\mathbb{D}$. Then, multiply $-u_\mathbb{D}''(x)=f(x)$ on both sides by some $v\in V$ (this $v$ is called a test function) and integrate to obtain
\begin{equation}\label{equiv1:eq1}
    -\int_0^1u_\mathbb{D}''(x)v(x)dx=\int_0^1f(x)v(x)dx.
\end{equation}
Using the formula for integration by parts,
\begin{equation}\label{equiv1:eq2}
    \int_0^1 a(x)b'(x)dx=a(1)b(1)-a(0)b(0)-\int_0^1a'(x)b(x)dx,
\end{equation}
with $a(x)=v(x)$ and $b(x)=u_\mathbb{D}'(x)$, and applying the fact that $v(0)=v(1)=0$ we get
\begin{equation}\label{equiv1:eq3}
    \int_0^1u_\mathbb{D}''(x)v(x)dx=-\int_0^1 v'(x)u_\mathbb{D}'(x)dx.
\end{equation}
Then, replacing \eqref{equiv1:eq3} on equation \eqref{equiv1:eq1} we get that
\begin{equation}\label{equiv1:eq4}
    \int_0^1 v'(x)u_\mathbb{D}'(x)dx=\int_0^1f(x)v(x)dx.
\end{equation}
Notice that \eqref{equiv1:eq4} is the equation associated to the variational problem $\mathbb{V}$ (see \eqref{eq:var}), and as $v\in V$ was arbitrary, we have that $u_\mathbb{D}$ satisfies \eqref{equiv1:eq4} (and so, \eqref{eq:var}) for all $v\in V$. Therefore, $u_\mathbb{D}$ is also a solution for $\mathbb{V}$.\checkmark\enter
On the other hand, let $u_\mathbb{V}\in V$ be the solution for $\mathbb{V}$ (on \cite{Lagrange}, it is shown that the solution for $\mathbb{V}$ is unique). Then, we have by \eqref{eq:var} that, for all $v\in V$,
\begin{equation}\label{equiv1:eq5}
    \int_0^1u_\mathbb{V}'(x)v'(x)dx-\int_0^1f(x)v(x)dx=0.
\end{equation}
Now, applying integration by parts \eqref{equiv1:eq2} in the same way as before, we obtain \eqref{equiv1:eq3}. Replacing \eqref{equiv1:eq3} on \eqref{equiv1:eq5} and unifying the integral, we get
\begin{equation}\label{equiv1:eq6}
    -\int_0^1[u_\mathbb{V}''(x)+f(x)]v(x)dx=0
\end{equation}
for all $v\in V$. Then, by \eqref{theo:1} we have that $u_\mathbb{V}''(x)+f(x)=0$ for $x\in(0,1)$. In other words,
\begin{equation}\label{equiv1:eq7}
    -u_\mathbb{V}''(x)=f(x),\ x\in(0,1).
\end{equation}
Notice that \eqref{equiv1:eq7}, along with the fact that $u_\mathbb{V}(0)=u_\mathbb{V}(1)=0$, is the equation associated to the differential problem $\mathbb{D}$ (see \eqref{eq:D}). Therefore, $u_\mathbb{V}$ is also a solution for $\mathbb{V}$ as long as $u''(x)$ exists and is continuous (regularity assumption). But, this last condition for $u_\mathbb{V}$ holds, as seen on \cite{Lagrange}, so, the result holds.\checkmark\enter
To complete the proof, we have to prove \eqref{theo:1}, which is exercise 1.1 from \cite{Lagrange}:
\begin{equation}\label{theo:1}
    \begin{split}
        \text{If $w$ is continuous on $[0,1]$ and}\\
        \int_0^1w(x)v(x)dx=0 \text{ for all }v\in V,\\
        \text{then }w(x)=0 \text{ for all }x\in(0,1).
    \end{split}
\end{equation}
\textit{\textbf{Proof:}} By contradiction, suppose that $w(x)\not=0$ for some $x\in(0,1)$. Then, $w(x_0)=c\in\mathbb{R}$ for $x_0\in(0,1)$. Without loss of generality, suppose that $c>0$ (for $c<0$ the argument is analogous). As $w$ is continuous, there is an interval centered at $x_0$, $I=(x_0-\delta,x_0+\delta)$, such that $f(x)>0$ for $x\in I$. Now, define $$v(x)=\left\{\begin{split}
    0,\ if\ &x\not\in I,\\
    \frac{c}{\delta}[x+(\delta-x_0)],\ if\ &x\in(x_0-\delta,x_0)\\
    -\frac{c}{\delta}[x-(\delta+x_0)],\ if\ &x\in(x_0,x_0+\delta),
\end{split}\right.$$
which is a function sketched on Figure \ref{fig:theo1}.
\begin{figure}[h!]
    \centering
    \includegraphics[scale=0.5]{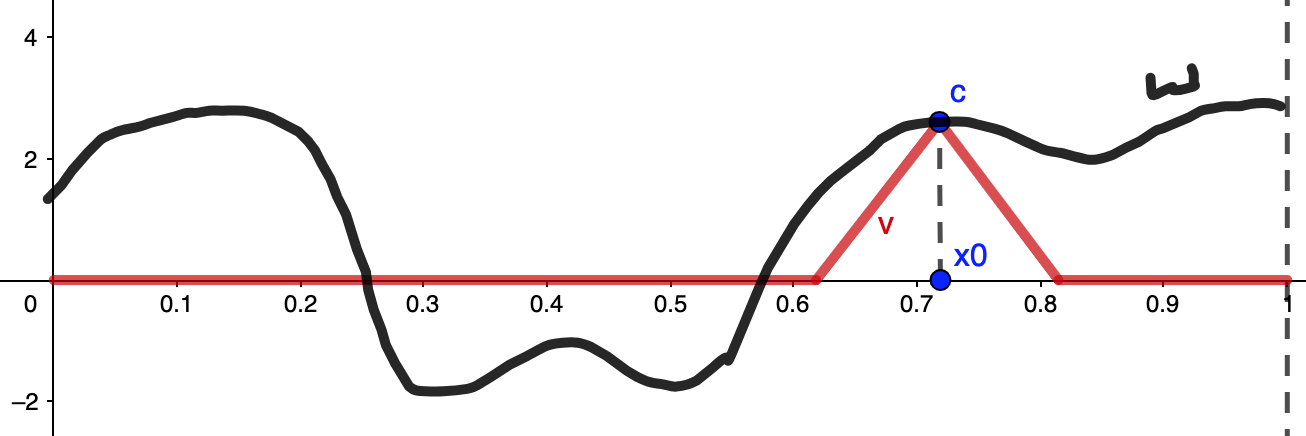}
    \caption{Sketch of the proof for \eqref{theo:1} involving an arbitrary function $w(x)$ (solid black line) and the built function $v\in V$ (solid red line).}
    \label{fig:theo1}
\end{figure}\\
Notice that $v(0)=v(1)=0$ because there is always a $\delta>0$ such that $0<x_0-\delta$ and $x_0+\delta<1$, and $v(x)=0$ for $x\not\in I$. Also, $v(x)$ is continuous on $[0,1]$ by construction (it was built with 4 lines that have connections in $(x_0-\delta,0),(x_0,c)$ and $(x_0+\delta,0)$). Moreover, $$v'(x)=\left\{\begin{split}
    0,\ if\ &x\not\in I,\\
    \frac{c}{\delta},\ if\ &x\in(x_0-\delta,x_0),\\
    -\frac{c}{\delta},\ if\ &x\in(x_0,x_0+\delta),
\end{split}\right.$$
is clearly piecewise continuous and bounded on $[0,1]$. Therefore, $v\in V$.\enter
However, as $w(x)v(x)>0$ for $x\in I$ and $w(x)v(x)=0$ for $x\not\in I$, then$$\int_0^1w(x)v(x)dx=\int_{x_0-\delta}^{x_0+\delta}w(x)v(x)dx>0.$$
In other words, if $w(x)\not=0$ for some $x\in (0,1)$, then we found $v\in V$ such that $\int_0^1w(x)v(x)dx\not=0$, which contradicts the hypothesis that $\int_0^1w(x)v(x)dx=0$ for all $v\in V$. Therefore, $w(x)=0$ for all $x\in(0,1)$.\qed
\enter
\texttt{Equivalence }$\mathbb{V}\iff\mathbb{M}$\\
Let $u_\mathbb{V}\in V$ be the solution for $\mathbb{V}$. Take some $v\in V$ and set $w=v-u_\mathbb{V}\in V$. Then
\begin{equation*}
    \begin{split}
        &F(v)=F(u+w)\\
        =&\frac{1}{2}\int_0^1\left[u_\mathbb{V}'(x)+w'(x)\right]^2dx-\int_0^1f(x)[u_\mathbb{V}(x)+w(x)]dx\\
        =&\frac{1}{2}\left(\int_0^1[u_\mathbb{V}'(x)]^2dx+\int_0^12u_\mathbb{V}'(x)w'(x)dx+\int_0^1[w'(x)]^2dx\right)\\
        &-\int_0^1f(x)u_\mathbb{V}(x)dx-\int_0^1f(x)w(x)dx\\
        =&\frac{1}{2}\int_0^1[u_\mathbb{V}'(x)]^2dx-\int_0^1f(x)u_\mathbb{V}(x)dx+\int_0^1u_\mathbb{V}'(x)w'(x)dx\\
        &-\int_0^1f(x)w(x)dx+\frac{1}{2}\int_0^1[w'(x)]^2dx\\
        \stackrel{\eqref{equiv1:eq4}}{=}&\frac{1}{2}\int_0^1[u_\mathbb{V}'(x)]^2dx-\int_0^1f(x)u_\mathbb{V}(x)dx+\frac{1}{2}\int_0^1[w'(x)]^2dx\\
        =&F(u_\mathbb{V})+\frac{1}{2}\int_0^1[w'(x)]^2dx\geq F(u_\mathbb{V}).
    \end{split}
\end{equation*}
In other words, we have that
\begin{equation}\label{equiv2:eq1}
    F(v)\geq F(u_\mathbb{V}).
\end{equation}
Notice that \eqref{equiv2:eq1} is the equation associated to the optimization problem $\mathbb{M}$ (see \eqref{eq:opt}), and as $v\in V$ was arbitrary, we have that $u_\mathbb{V}$ satisfies \eqref{equiv2:eq1} (and so, \eqref{eq:opt}) for all $v\in V$. Therefore, $u_\mathbb{V}$ is also a solution for $\mathbb{M}$.\checkmark\enter
On the other hand, let $u_\mathbb{M}\in V$ be a solution for $\mathbb{M}$. Then, for any $v\in V$ and any $\alpha\in\mathbb{R}$ we have that $u_\mathbb{M}+\alpha v\in V$, and so $F(u_\mathbb{M})\leq F(u_\mathbb{M}+\alpha v)$ by \eqref{eq:opt}. Therefore, the minimum for $F(u_\mathbb{M}+\alpha v)$ is achieved at $\alpha=0$.\enter
Now, define $g(\alpha)=F(u_\mathbb{M}+\alpha v)$, which is a differentiable function \cite{Lagrange}. We have that
\begin{equation}\label{equiv2:eq2}
\begin{split}
    &g(\alpha)=F(u_\mathbb{M}+\alpha v)\\
    =&\frac{1}{2}\int_0^1\left[u_\mathbb{M}'(x)+\alpha v'(x)\right]^2dx-\int_0^1f(x)[u_\mathbb{M}(x)+\alpha v(x)]dx\\
    =&\frac{1}{2}\int_0^1[u_\mathbb{M}'(x)]^2dx+\alpha\int_0^1u_\mathbb{M}'(x)v'(x)dx+\frac{\alpha^2}{2}\int_0^1[v'(x)]^2dx\\
    &-\int_0^1f(x)u_\mathbb{M}(x)dx-\alpha\int_0^1f(x)v(x)dx.
\end{split}
\end{equation}
Since $g(\alpha)=F(u_\mathbb{M}+\alpha v)$ has a minimum at $\alpha=0$, then $g'(0)=0$. Therefore, taking the derivative on \eqref{equiv2:eq2} and replacing $\alpha=0$, we get
\begin{equation}\label{equiv2:eq3}
\begin{split}
    &g'(\alpha)=\int_0^1u_\mathbb{M}'(x)v'(x)dx+\alpha\int_0^1[v'(x)]^2dx-\int_0^1f(x)v(x)dx\\
    \implies&g'(0)=\int_0^1u_\mathbb{M}'(x)v'(x)dx-\int_0^1f(x)v(x)dx.
\end{split}
\end{equation}
Finally, as $g'(0)=0$, we get from \eqref{equiv2:eq3}, that
\begin{equation}\label{equiv2:eq4}
    \int_0^1u_\mathbb{M}'(x)v'(x)dx=\int_0^1f(x)v(x)dx.
\end{equation}
Notice that \eqref{equiv2:eq4} is the equation associated to the variational problem $\mathbb{V}$ (see \eqref{eq:var}), and as $v\in V$ was arbitrary, we have that $u_\mathbb{M}$ satisfies \eqref{equiv2:eq4} (and so, \eqref{eq:var}) for all $v\in V$. Therefore, $u_\mathbb{M}$ is also a solution for $\mathbb{V}$.\checkmark\enter
As mentioned before, on \cite{Lagrange}, it is shown that the solution for $\mathbb{V}$ is unique. Therefore, as $\mathbb{V}$ and $\mathbb{M}$ are equivalent, then the solution for $\mathbb{M}$ is also unique.
\subsubsection{Lagrange finite elements}\label{sec:Lagrange}
In this section we consider the Laplace equation \cite{Lagrange},
\begin{equation}\label{Laplace}
\begin{split}
        -\Delta &p=f\text{ in } \Omega,\\
        &p=0\text{ in }\Gamma,
\end{split}
\end{equation}
where $\Omega\subseteq\mathbb{R}^2$ is an open-bounded domain with boundary $\Gamma$, $f$ is a given function and $\Delta p=\frac{\partial^2p}{\partial x^2}+\frac{\partial^2p}{\partial y^2}$. Following \cite{Lagrange},
consider the space
\begin{equation*}
    V=\{v:v \text{ continuous on } \Omega, \frac{\partial v}{\partial x},\frac{\partial v}{\partial y} \text{ piecewise continuous on }\Omega\text{ and } v=0 \text{ on } \Gamma\}.
\end{equation*}
Alternatively we can work in the Sobolev space (see \cites{Lagrange, Galvis, Mixed})
\begin{equation*}
    H^1(\Omega)=\{v\in L^2(\Omega)\ \Big|\ \frac{\partial v}{\partial x},\frac{\partial v}{\partial y}\in L^2(\Omega)\}.
\end{equation*}
Here $L^2(\Omega)=\{v:\Omega\rightarrow\mathbb{R}\ \Big|\int_\Omega v^2<\infty\}$.\\

We multiply the first equation of \eqref{Laplace} by some $v\in V$ (referred to as \textit{test function}) and integrate over $\Omega$ to obtain
\begin{equation}\label{Laplace2}
    -\int_\Omega \Delta p\ v=\int_\Omega f\ v.
\end{equation}
Applying divergence theorem we obtain the Green's formula (\cite{Lagrange}),
\begin{equation}\label{Green}
    -\int_\Omega \Delta p\ v=\int_\Omega\nabla v\cdot\nabla p-\int_\Gamma v\ \nabla p\cdot\eta,
\end{equation}
where $\eta$ is the outward unit normal to $\Gamma$. Since $v=0$ on $\Gamma$, the third integral equals $0$. Note that the boundary integral does not depend on $p$'s value on $\Gamma$ but rather on the normal derivative of $p$ in $\Gamma$. Due to this fact the 
boundary condition $p=0$ on $\Gamma$ is know as an \textbf{essential} boundary condition.\\

Then, replacing \eqref{Green} on \eqref{Laplace2}, we get,
\begin{equation}\label{Laplace3}
    \int_\Omega\nabla v\cdot\nabla p=\int_\Omega f\ v.
\end{equation}
This holds for all $v\in V$. This is called weak formulation of the Laplace equation 
\eqref{Laplace}. We remark that, according to \cite{Lagrange},  if $p\in V$ satisfies \eqref{Laplace3} for all $v\in V$ and is sufficiently regular, then $p$ also satisfies \eqref{Laplace}, i.e., it's a (classical) solution for our problem. For more details see \cite{Lagrange} and references therein. \\

In order to set the problem for a computer to solve it, we are going to discretize it and encode it into a linear system.\enter
First, consider a \textit{triangulation} $T_h$ of the domain $\Omega$. This is, $T_h=\{K_1,\dots,K_m\}$ a set of non-overlapping triangles such that $\Omega=K_1\cup\dots\cup K_m$ and no vertex ($N_i$) of one triangle lies on the edge of another triangle, as seen on Figure \ref{fig:Triangulation}. 
\newpage
\begin{figure}[h!]
    \centering
    \includegraphics[scale=0.29]{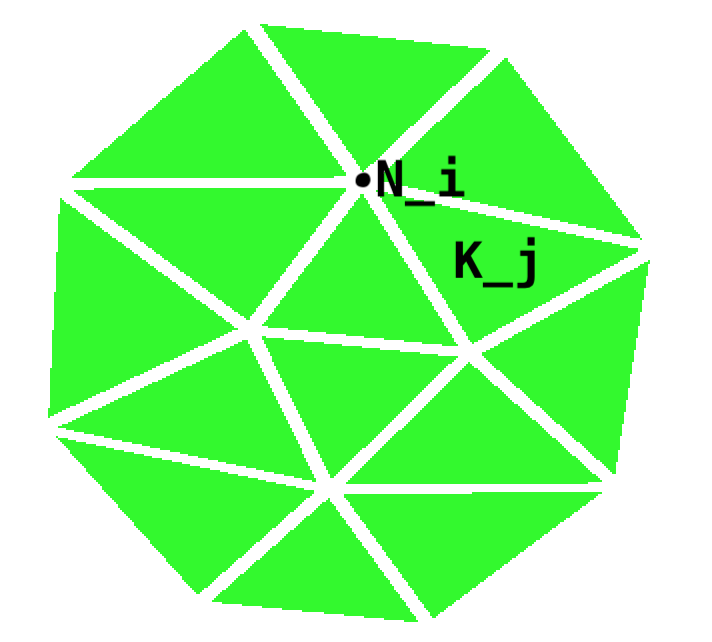}
    \caption{A triangulation for a given domain $\Omega$ showing a node $N_i$, and formed by some triangles $K_j$. \textit{Note:} Triangles have been separated in the edges to take a better look, but the triangulation has no empty spaces. Visualization: \cite{Glvis}.}
    \label{fig:Triangulation}
\end{figure}

The $h$ in the notation $T_h$ is a measure of the size of mesh, it usually refers to a typical element diameter or perhaps to the largest element diameter in the triangulation. In this manuscript $h$ is defined by  $h=\max\{\mbox{diam}(K):K\in T_h\}$ where $\mbox{diam}(K)=\text{longest side of }K$.\\
\newline
Now, let $V_h=\{v:v \text{ continuous on }\Omega, v|_K \text{ linear for }K\in T_h,\ v=0\text{ on }\Gamma\}$.\\
We consider the nodes ($N_1,\dots,N_M$) of the triangulation that are not on the boundary, because $p=0$ there, and we define some functions $\varphi_j\in V_h$ in such way that \[\varphi_j(N_i)=\left\{ \begin{array}{lcc}
             1 &   ,\ i=j\\
             \\0 &  ,\ i\not=j \\
             \end{array}
   \right.\]
for $i,j=1,\dots,M$. See Figure  \ref{fig:ShapeFunc} for an illustration of 
$\varphi_j\in V_h$.
\begin{figure}[h!]
    \centering
    \includegraphics[scale=0.29]{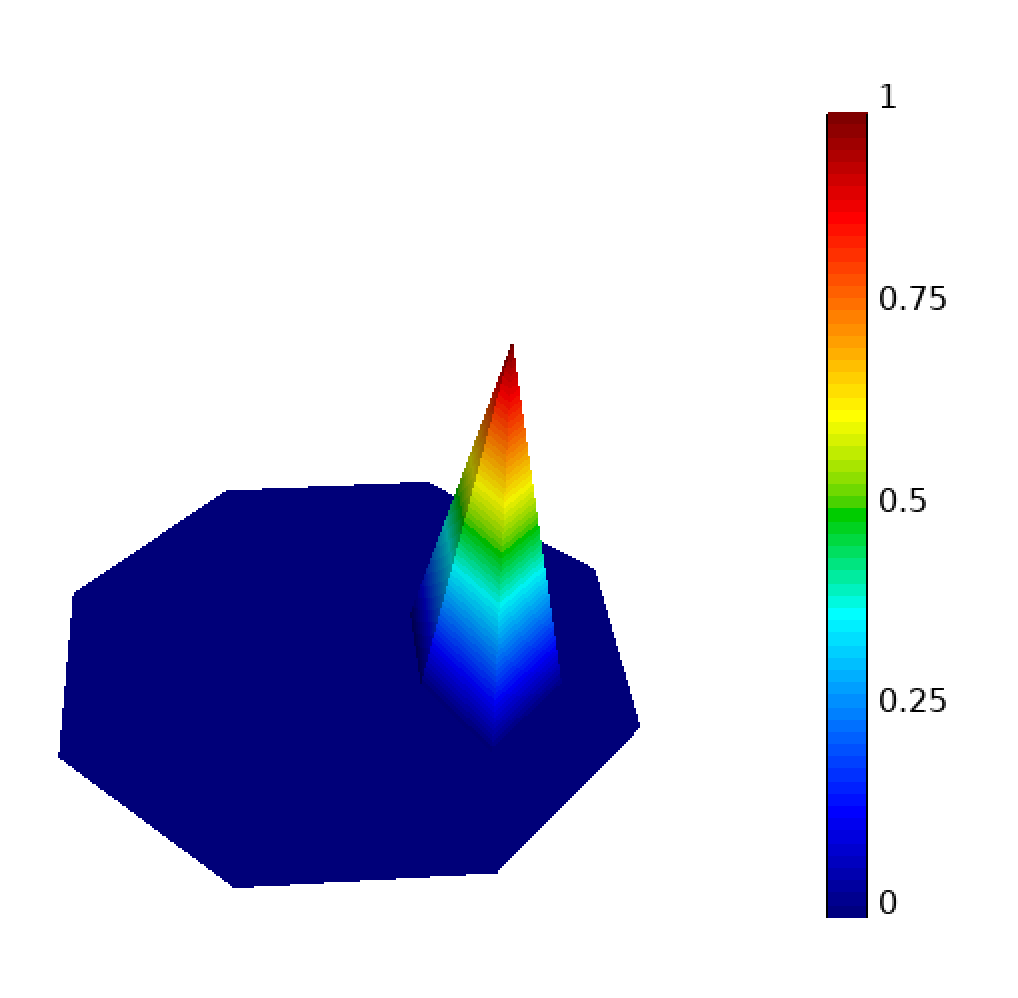}
    \includegraphics[scale=0.29]{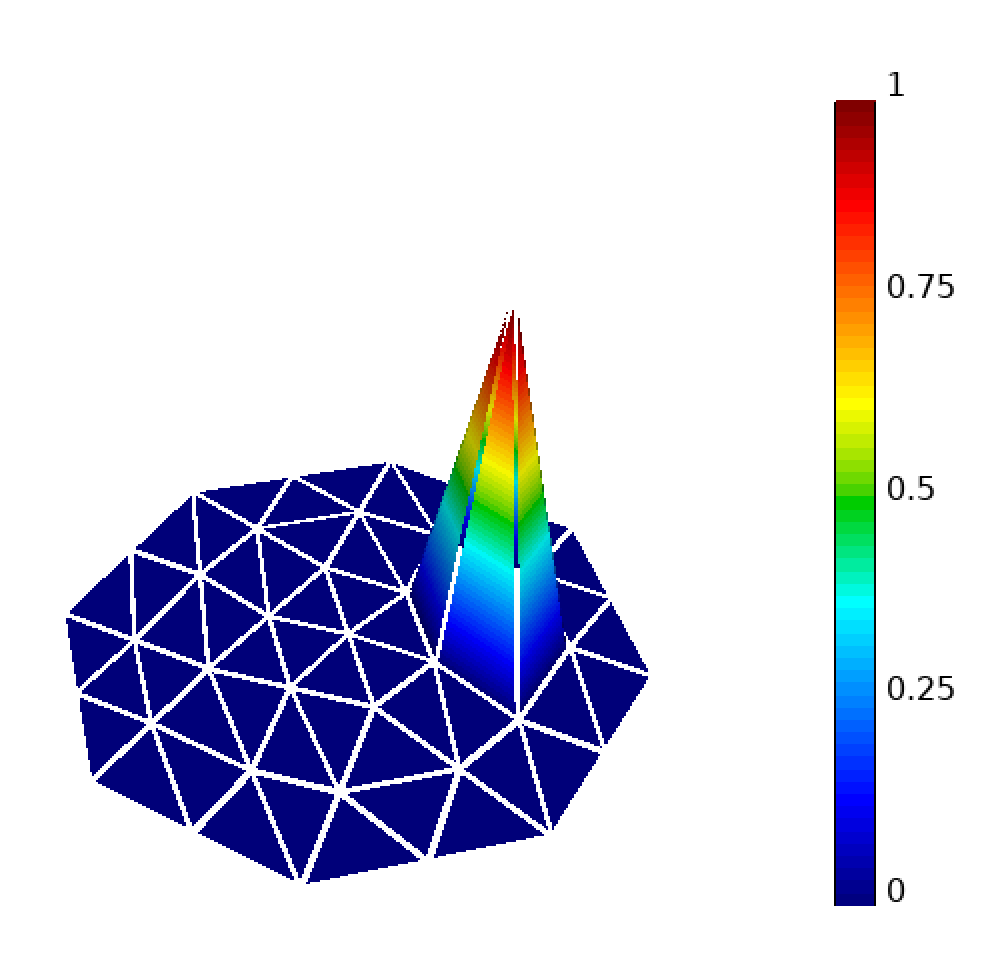}
    \caption{Illustration of the function $\varphi_j$ produced with MFEM library. On the left picture we plot the function $\varphi_j$. On the right picture we show the same plot depicting the elements of the underlying triangulation. Visualization: \cite{Glvis}.}
    \label{fig:ShapeFunc}
\end{figure}
\newpage
With this, $V_h=\mbox{span}\{\varphi_i:i=1,\dots,M\}$ and for any given $v\in V_h$ we have $v(x)=\sum_{j=1}^M\xi_j\varphi_j(x),$ with $\xi_j=v(N_j)$ and $x\in\Omega\cup\Gamma$. So, $V_h$ is a finite-dimensional subspace of $V$. See \cite{Lagrange} for details. \\
\newline
Then, if $p_h\in V_h$ satisfies \eqref{Laplace3} for all $v\in V_h$, in particular,
\begin{equation}\label{LaplaceD1}
    \int_\Omega\nabla p_h\cdot\nabla \varphi_j=\int_\Omega f\ \varphi_j,\ \ j=1,\dots,M.
\end{equation}
Since $\nabla p_h=\sum_{i=1}^M\xi_i\nabla\varphi_i$ with $\xi_i=p_h(N_i)$, replacing on \eqref{LaplaceD1} we get,
\begin{equation}\label{LaplaceD2}
    \sum_{i=1}^M\xi_i\int_\Omega\nabla\varphi_i\cdot\nabla\varphi_j=\int_\Omega f\ \varphi_j,\ \ j=1,\dots,M.
\end{equation}
Finally, \eqref{LaplaceD2} is a linear system of $M$ equations and $M$ unknowns ($\xi_1,\dots,\xi_M$), which can be written as,
\begin{equation}\label{LaplaceEnd}
    A\xi=b,
\end{equation}
where $A[i,j]=\int_\Omega\nabla\varphi_i\cdot\nabla\varphi_j$, $\xi[i]=p_h(N_i)$ and $b[i]=\int_\Omega f\ \varphi_i$.\\
\newline
We can solve \eqref{LaplaceEnd} with MFEM library as done on Section \ref{sec:Vs} (Example\#1). Before continuing with the next section, let us show some theorems regarding the error between the solution $p$ for problem \eqref{Laplace} and its approximation $p_h$. The theorems are presented on the general form but, after the proof, we show how is it used on our particular problem.\enter
For the following theorem, $A$ is a bilinear form on $V\times V$ and $L$ is a linear form on $V$ such that 
\begin{enumerate}
    \item \textbf{$A$ is continuous} ($\mathcal{C}$)\\ There is a constant $\gamma>0$ such that $$|A(v,w)|\leq\gamma||v||_V||w||_V,\ \forall v,w\in V.$$
    \item \textbf{$A$ is $V$-elliptic} ($V_\epsilon$)\\ There is a constant $\alpha>0$ such that $$\alpha||v||_V^2\leq A(v,v).$$
\end{enumerate}
\begin{theorem}\label{theo:cea}\cite{Lagrange}\ \textbf{C\'ea Lemma}\\
If $p\in V$ is the solution for $$A(p,v)=L(v),\ \forall v\in V$$
and $p_h\in V_h\subset V$ is the solution for $$A(p_h,v_h)=L(v_h),\ \forall v_h\in V_h$$
then, $$||p-p_h||_V\leq\frac{\gamma}{\alpha}||p-v_h||_V,\ \forall v_h\in V_h.$$
\end{theorem}
\begin{proof}
Using the hypothesis that $A(p,v)=L(v)$ for all $v\in V$, along with the fact that $V_h\subset V$, we have that $A(p,v_h)=L(v_h)$ for all $v_h\in V_h$. Now, we subtract the last equation with the one given as hypothesis, $A(p_h,v_h)=L(v_h)$, to get, $A(p-p_h,v_h)=L(v_h)-L(v_h)=0$ for all $v_h\in V_h$. \enter
For an arbitrary $w\in V_h$, let $v_h=p_h-w\in V_h$. Then, 
\begin{equation*}
    \begin{split}
        &\alpha||p-p_h||^2_V\\
        \leq&A(p-p_h,p-p_h)+0\quad(\text{See } V_\epsilon)\\
        =&A(p-p_h,p-p_h)+A(p-p_h,w)\\
        =&A(p-p_h,p-p_h+w)\\
        =&A(p-p_h,p-p_h+p_h-v_h)\\
        =&A(p-p_h,p-v_h)\\
        \leq&\gamma||p-p_h||_V||p-v_h||_V\quad(\text{See } \mathcal{C}).
    \end{split}
\end{equation*}
In other words, we have that $\alpha||p-p_h||^2_V\leq\gamma||p-p_h||_V||p-v_h||_V$ for all $v_h\in V_h$. Dividing by $\alpha||p-p_h||_V$ on both sides, we get, $||p-p_h||_V\leq\frac{\gamma}{\alpha}||p-v_h||_V$ for all $v_h\in V_h.$ Note that $||p-p_h||_V\not=0$ because $p_h$ is supposed to be an approximation for $p$, and not the exact solution.
\end{proof}
On the particular case of \eqref{Laplace3}, we have that $A(p,v)=\int_\Omega\nabla v\cdot\nabla p$ and $L(v)=\int_\Omega f\ v.$ Using Cauchy-Schwarz Inequality for Integrals we have that 
$$A(p,v)=\int_\Omega\nabla v\cdot\nabla p\leq\sqrt{\int_\Omega|\nabla p|^2}\cdot\sqrt{\int_\Omega|\nabla v|^2}=1\cdot||p||_V||v||_V.$$
In other words, the parameter for the continuity of the bilinear form of our particular case is $\gamma=1$. Also, notice that $$1\cdot||v||_V^2=\left(\sqrt{\int_\Omega|\nabla v|^2}\right)^2=\int_\Omega|\nabla v|^2=\int_\Omega\nabla v\nabla v=A(v,v).$$ That is, the parameter for the $V$-ellipticity of the bilinear form for our particular case is $\alpha=1$. Therefore, C\'ea Lemma ensures that $\sqrt{\int_\Omega|p-p_h|^2}\leq\sqrt{\int_\Omega|p-v_h|^2}$ for all $v_h\in V_h$.\enter
Now, before presenting the second theorem, define the operator $\mathcal{I}^h:\mathcal{C}(\Omega)\rightarrow V_h$ that associates every continuous function whose domain is $\Omega$, $f\in\mathcal{C}(\Omega)$, with a function $\mathcal{I}^hf\in V_h$ \cite{Galvis}. This operator is an interpolation operator defined by the nodes of the triangulation of $\Omega$: if $T_h=\{K_1,\dots,K_m\}$ is a triangulation of $\Omega$, then $\mathcal{I}^hf(K_i)=f(K_i),\ i=1,\dots,m$.
\begin{theorem}\label{theo:errorLagrange}\cite{Galvis}
Let $\Omega\subseteq\mathbb{R}^2$ be a polygonal domain. Let $\{T_{h_i}\}$ be a family of triangulations of $\Omega$, with $T_{h_i}$ being a quasi-uniform triangulation. Then, $$||\mathcal{I}^hp-p||_1\leq c\cdot h\cdot||p||_2,$$
where $$||f||_2=\left(\int_\Omega f(x)^2+|\nabla f(x)|^2+\sum_{ij}(\partial_{ij}f)^2dx\right)^{1/2}.$$
\end{theorem}
The proof of this theorem is out of the scope for this work. However, it can be checked on \cite{Galvis}. In summary, for our case, C\'ea Lemma states that Laplace problem can be approximated by the space $V_h$, and Theorem \ref{theo:errorLagrange} states that the approximation is a good one.

\subsubsection{Lagrange spaces of higher order}\label{sec:HighOrder}
This short section has the purpose of explaining Lagrange finite element spaces of higher order. Previously, on Section \ref{sec:Lagrange}, when introducing Lagrangian elements, the shape function's degree was set to one. Better approximations can be obtained by using polynomials of higher order. One can define, for a fixed order $k$, 
\begin{equation*}
\begin{split}
    V^k_h=\{v:&v \text{ continuous on }\Omega,\\ &v|_K \text{ polynomial of order at most }k,K\in T_h,\ v=0\text{ on }\Gamma\}.
\end{split}
\end{equation*}

For example, as seen in \cite{Galvis}, the space of Bell triangular finite elements for a given triangulation $T_h$ is the space of functions that are polynomials of order 5 when restricted to every triangle $K\in T_h$. That is, if $v$ is in this space, then,
\begin{equation*}
    v|_K(x,y)=a_1x^5+a_2y^5+a_3x^4y+a_4xy^4+\dots+a_{16}x+a_{17}y+a_{18}
\end{equation*}
for all $K\in T_h$. Here, the constants $a_i,\ i=1,\dots,18$ correspond to $v$'s \textit{DOF} (degrees of freedom).\\
\newline
On Figures \ref{fig:HighOrderFunc} and \ref{fig:HighOrderFunc2}, we present a visualization of some shape functions of different orders. We encourage the reader to compare them with Figure \ref{fig:ShapeFunc} and notice the degree of the polynomial in the nonzero part of the shape functions.
\begin{figure}[h!]
    \centering
    \includegraphics[scale=0.3]{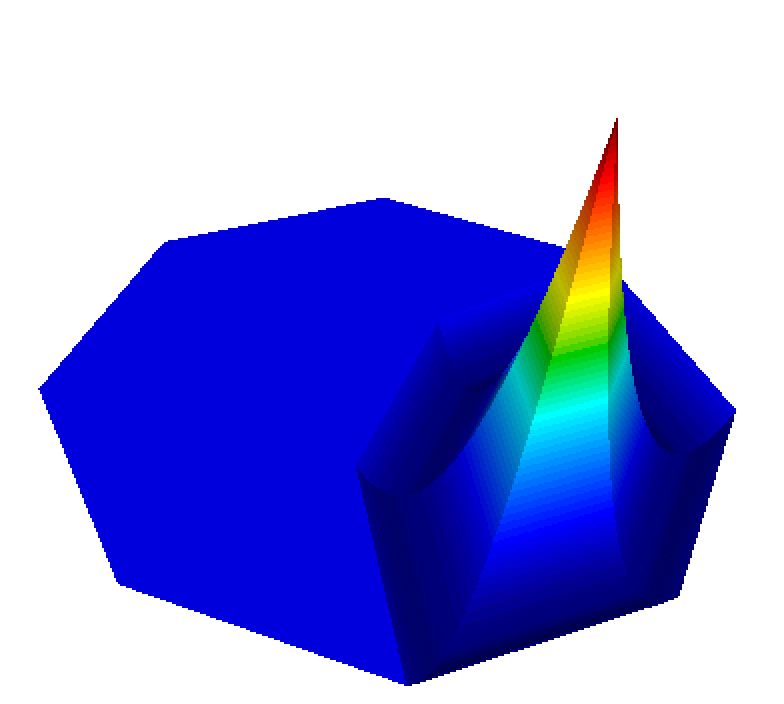}
    \includegraphics[scale=0.28]{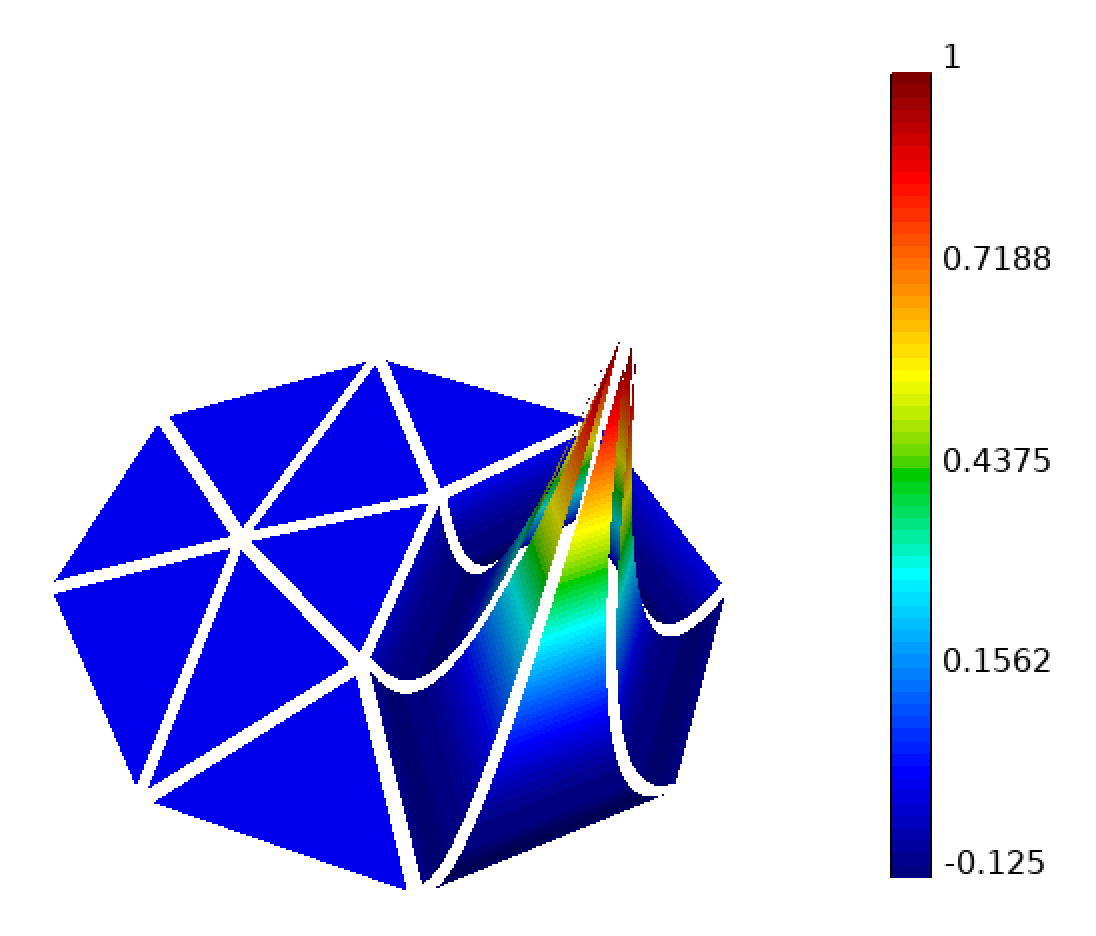}
    \caption{Illustration of finite element basis (shape) functions of order 2. On the left picture we show one continuous basis function. On the right we also show the underlying triangulation. Visualization: \cite{Glvis}.}
    \label{fig:HighOrderFunc}
\end{figure}
\begin{figure}[h!]
    \centering
    \includegraphics[scale=0.35]{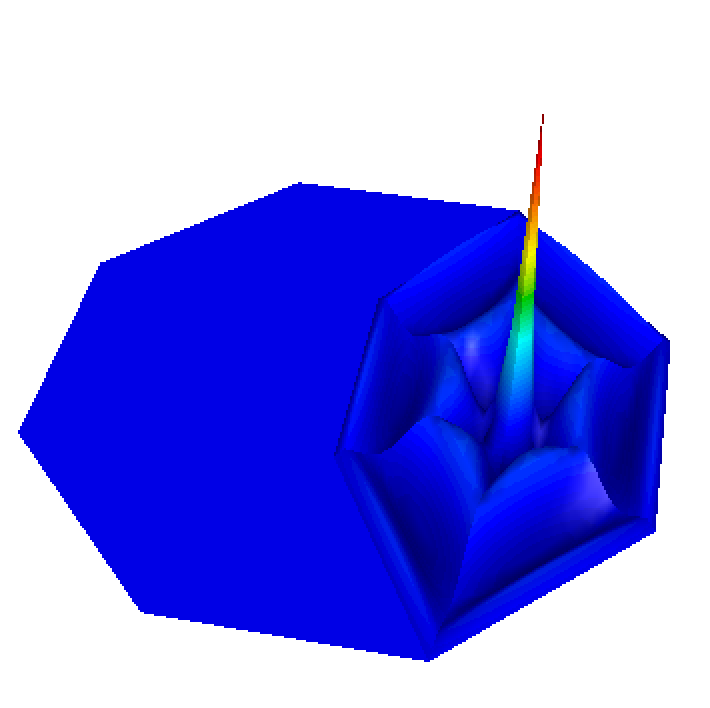}
    \includegraphics[scale=0.3]{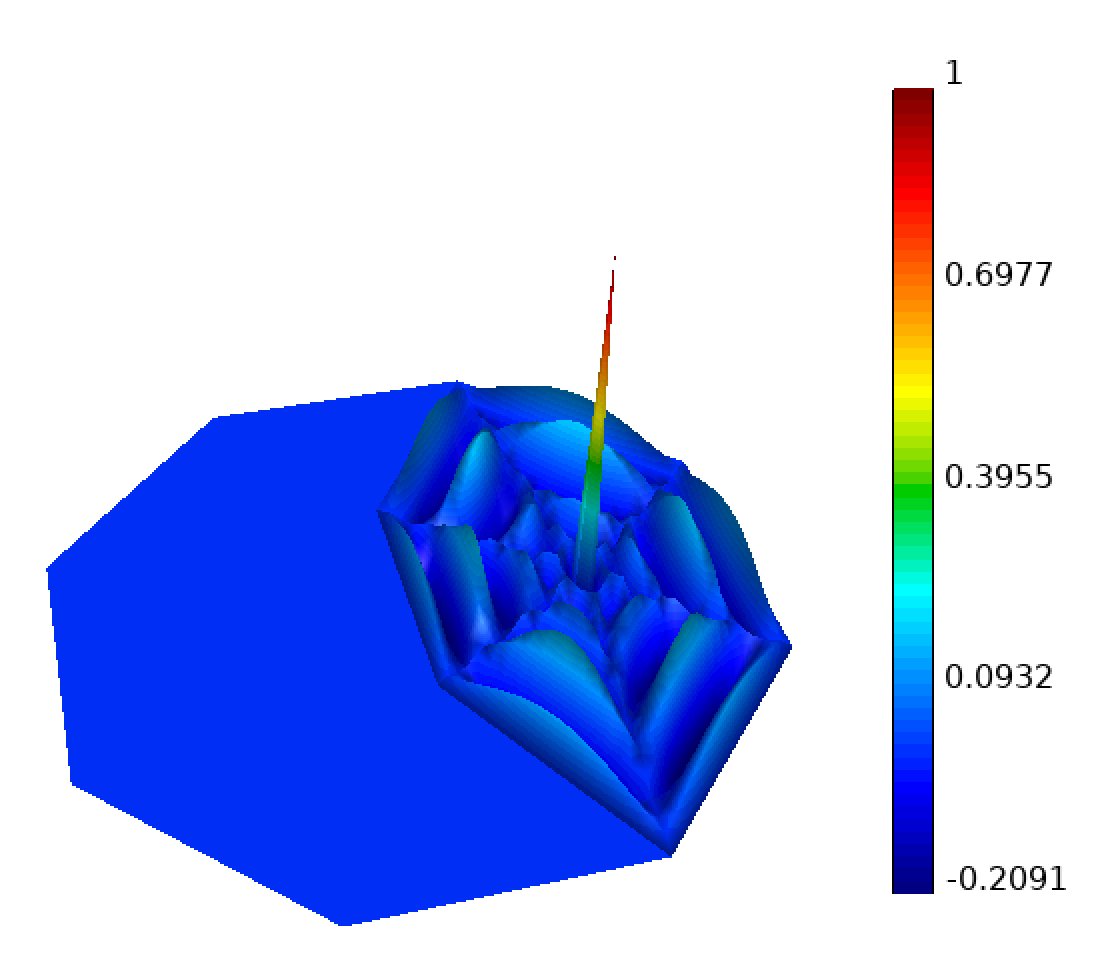}
    \caption{Illustration of finite elements basis function of orders 5 (left) and 10 (right). Visualization: \cite{Glvis}.}
    \label{fig:HighOrderFunc2}
\end{figure}

\newpage

\subsubsection{Raviart-Thomas finite elements}\label{sec:Mixed}
First, let's define some important spaces, where $\Omega$ is a bounded domain in $\mathbb{R}^2$ and $\Gamma$ its boundary. See \cites{Lagrange,Mixed, Galvis} and references therein for details. 
The space of all square integrable functions,
\begin{equation*}
    L^2(\Omega)=\{v:\Omega\rightarrow\mathbb{R}\ \Big|\int_\Omega v^2<\infty\}.
\end{equation*}
We also use the first order Sobolev space, 
\begin{equation*}
    H^1(\Omega)=\{v\in L^2(\Omega)\ \Big|\ \frac{\partial v}{\partial x},\frac{\partial v}{\partial y}\in L^2(\Omega)\}
\end{equation*}
and the subspace of $H^1(\Omega)$ of functions with vanishing value on the boundary, 
\begin{equation*}
    H_0^1(\Omega)=\{v\in H^1(\Omega)\ |\ v=0\ on\ \Gamma\}.
\end{equation*}
We also introduce the  space of square integrable vector functions with square integrable divergence, 
\begin{equation*}
    H(\mbox{div};\Omega)=\{\mathbf{v}\in L^2(\Omega)\times L^2(\Omega)\ |\ \mbox{div}(\mathbf{v})\in L^2(\Omega)\}.
\end{equation*}
As above, let $\Omega\in\mathbb{R}^2$ be a bounded domain with boundary $\Gamma$ and consider problem \eqref{Laplace}. This time we require explicitly that 
 $f\in L^2(\Omega)$. Recall (from Section \ref{sec:Lagrange}) that this problem can be reduced to
\begin{equation*}
\int_\Omega\nabla v\cdot\nabla p=\int_\Omega f\ v,\text{ for all $v\in V$},
\end{equation*}
where Dirichlet boundary condition ($p=0\ in\ \Gamma$) is essential. Recall that 
we can take  $V=H_0^1(\Omega)$ as seen in \cite{Lagrange, Mixed}.\\

Let $u=\nabla p$ in $\Omega$. Then, problem \eqref{Laplace} can be written as the following system of fist order  partial differential equations,
\begin{equation}\label{LaplaceMixed}
    \begin{split}
    u&=\nabla p\text{ in }\Omega\\  
    \mbox{div}(u)&=-f\text{ in } \Omega\\  
        p&=0\text{ in }\Gamma,
    \end{split}
\end{equation}
because $\Delta p=\mbox{div}(\nabla p)$.
Now, following a similar procedure as in Section \ref{sec:Lagrange}, multiply the first equation of \eqref{LaplaceMixed} by some $\mathbf{v}\in H(\mbox{div};\Omega)$ and integrate both sides to obtain,
\begin{equation}\label{LaplaceMixed2}
    \int_\Omega u\ \mathbf{v}=\int_\Omega\nabla p\cdot\mathbf{v}.
\end{equation}
Consider Green's identity \cite{Mixed},
\begin{equation}\label{Green2}
    \int_\Omega \mathbf{v}\cdot\nabla p+\int_\Omega p\ \mbox{div}(\mathbf{v})=\int_\Gamma(\mathbf{v}\cdot\eta) p,
\end{equation}
where $\eta$ is the normal vector exterior to $\Gamma$.\\
\newline
Replacing \eqref{Green2} in \eqref{LaplaceMixed2}, and considering the third equation of \eqref{LaplaceMixed}, we get,
\begin{equation}\label{LaplaceMixed3}
    \int_\Omega u\ \mathbf{v}+\int_\Omega p\ \mbox{div}(\mathbf{v})=\int_\Gamma(\mathbf{v}\cdot\eta) p,
\end{equation}
On the other hand, we can multiply the second equation of problem \eqref{LaplaceMixed} by some $w \in L^2(\Omega)$, integrate and obtain,
\begin{equation}\label{LaplaceMixed4}
    \int_\Omega w\ \mbox{div}(u)=-\int_\Omega f\ w.
\end{equation}
Note that the boundary integral depends directly on the value of $p$ in $\Gamma$. And, this is referred to as the case of a \textbf{natural} boundary condition. Observe that the same boundary condition appeared as an essential boundary condition in the second order formulation considered before (Section \ref{sec:Lagrange}). In this first order formulation it showed up as a natural boundary condition. \\

Finally, applying boundary condition $p=0\ \text{in }\Gamma$ into \eqref{LaplaceMixed3}, and joining \eqref{LaplaceMixed3} and \eqref{LaplaceMixed4}. We get the following problem deduced from \eqref{LaplaceMixed},
\begin{equation}\label{LaplaceMixedVar}
    \begin{split}
        &\int_\Omega u\ \mathbf{v}+\int_\Omega p\ \mbox{div}(\mathbf{v})=0\\
        &\int_\Omega w\ \mbox{div}(u)=-\int_\Omega f\ w.
    \end{split}
\end{equation}
For this problem, which is  a variational formulation of \eqref{LaplaceMixed}, the objective is to find $(u,p)\in H(\mbox{div};\Omega)\times L^2(\Omega)$ such that it is satisfied for all $\mathbf{v}\in H(\mbox{div};\Omega)$ and all $w\in L^2(\Omega)$.\enter
For the discretized problem related to \eqref{LaplaceMixedVar}, in \cite{Mixed} the following spaces are defined for a  \textit{triangulation} $T_h$ of the domain $\Omega$ and a fixed integer $k\geq0$,
\begin{equation*}
    \begin{split}
        &H_h^k:=\{\mathbf{v_h}\in H(\mbox{div};\Omega)\ |\ \mathbf{v_h}|_K\in RT_k(K)\text{ for all }K\in T_h\},\text{ and}\\
        &L_h^k:=\{w_h\in L^2(\Omega)\ |\ w_h|_K\in \mathbb{P}_k(K)\text{ for all }K\in T_h\},
    \end{split}
\end{equation*}
where 
\begin{equation*}
    \begin{split}
        &\mathbb{P}_k(K)=\{p:K\rightarrow\mathbb{R}\ |\ p\text{ is a polynomial of degree }\leq k\},\text{ and}\\
        &RT_k(K)=[\mathbb{P}_k(K)\times\mathbb{P}_k(K)]+\mathbb{P}_k(K)x.
    \end{split}
\end{equation*}
Note that $\mathbf{p}\in RT_k(K)$ if and only if there are some $p_0,p_1,p_2\in\mathbb{P}_k(K)$ such that
\begin{equation*}
    \mathbf{p}(x)=\begin{pmatrix}
    p_1(x)\\p_2(x)
    \end{pmatrix}
    +p_0(x)\begin{pmatrix}
    x\\y
    \end{pmatrix}
    \text{ for all }\begin{pmatrix}
    x\\y
    \end{pmatrix}\in K,
\end{equation*}
and, also note that $\mathbf{p}$ has a degree of $k+1$.\\
\newline
Then, \eqref{LaplaceMixedVar} gives the following discrete problem: find $(u_h,p_h)\in H_h^k\times L_h^k$ such that
\begin{equation}\label{LaplaceMixedVar2}
    \begin{split}
        &\int_\Omega u_h\ \mathbf{v}_h+\int_\Omega p_h\ \mbox{div}(\mathbf{v}_h)=0\\
        &\int_\Omega w_h\ \mbox{div}(u_h)=-\int_\Omega f\ w_h,
    \end{split}
\end{equation}
for all $\mathbf{v}_h\in H_h^k$ and all $w_h\in L_h^k$.\\
\newline
As spaces $H_h^k$ and $L_h^k$ are finite dimensional, they have a finite basis. That is, $H_h^k=\mbox{span}\{\varphi_i:i=1,\dots,M\}$ and $L_h^k=\mbox{span}\{\psi_j:j=1,\dots,N\}$. Then, $u_h=\sum_{i=1}^Mu_i\varphi_i$ and $p_h=\sum_{j=1}^Np_j\psi_j$, where $u_i$ and $p_j$ are scalars.
\newpage
In particular, as $\varphi_k\in H_h^k$ and $\psi_l\in L_h^k$, we have that problem \eqref{LaplaceMixedVar2} can be written as,
\begin{equation}\label{LaplaceMixedVar3}
    \begin{split}
    &\int_\Omega\left(\sum_{i=1}^Mu_i\varphi_i\right)\varphi_k+\int_\Omega\left(\sum_{j=1}^Np_j\psi_j\right)\mbox{div}(\varphi_k)=0\\
    &\int_\Omega\psi_l\mbox{div}\left(\sum_{i=1}^Mu_i\varphi_i\right)=\int_\Omega f\psi_l,
    \end{split}
\end{equation}
for $k=1,\dots,M$ and $l=1,\dots,N$. Which is equivalent to the following, by rearranging scalars,
\begin{equation}\label{LaplaceMixedVar4}
    \begin{split}
    &\sum_{i=1}^Mu_i\int_\Omega\varphi_i\cdot\varphi_k+\sum_{j=1}^Np_j\int_\Omega\psi_j\mbox{div}(\varphi_k)=0\\
    &\sum_{i=1}^Mu_i\int_\Omega\psi_l\mbox{div}(\varphi_i)=\int_\Omega f\psi_l,
    \end{split}
\end{equation}
for $k=1,\dots,M$ and $l=1,\dots,N$. The problem \eqref{LaplaceMixedVar4} can be formulated into the following matrix system
\begin{equation}\label{LaplaceMixedEnd}
    \begin{pmatrix}
    A & B\\
    B^t & 0
    \end{pmatrix}\begin{pmatrix}
    U\\
    P
    \end{pmatrix}=\begin{pmatrix}
    0\\
    F
    \end{pmatrix},
\end{equation}
where $A$ is a $N\times N$ matrix, $B$ is a $M\times N$ matrix with $B^t$ it's transpose, $U$ is a $M$-dimensional column vector and $P,F$ are $N$-dimensional column vectors. \\
The entries of these arrays are $A[i,j]=\int_\Omega\varphi_i\cdot\varphi_j$, $B[i,j]=\int_\Omega\psi_j\mbox{div}(\varphi_i)$, $U[i]=u_i$, $P[i]=p_i$ and $F[i]=\int_\Omega f\psi_i$.\\
\newline 
The linear system \eqref{LaplaceMixedEnd} can be solved for $(U,P)$ with a computer using MFEM library. Note that with the entries of $U$ and $P$, the solution $(u_h,p_h)$ of \eqref{LaplaceMixedVar2} can be computed by their basis representation.\\

 The spaces defined to discretize the problem are called Raviart-Thomas finite element spaces. The fixed integer $k$ is also called the order of the shape functions or the order of the finite element space. The parameter $h$ is the same as in Section \ref{sec:Lagrange}, which is a measure of size for $T_h$. 
 See \cite{Mixed} for details.

\newpage

\subsubsection{Taylor-Hood finite elements}\label{sec:NSFEM}
In this section, we show the spatial discretization done in \cite{TH}, for Stokes equations \eqref{eq:THStokes},
\begin{equation}\label{eq:THStokes}
    \begin{split}
        \frac{\partial u}{\partial t}-\nu\Delta u+\nabla p=f,\text{ in }\Omega,\\
        \nabla\cdot u = 0,\text{ in }\Omega,\\
        u=g,\text{ in } \Gamma,
    \end{split}
\end{equation}
and then mention the corresponding spatial discretization for Navier-Stokes equations \eqref{eq:THNS}, as an extension of the previous one,
\begin{equation}\label{eq:THNS}
    \begin{split}
        \frac{\partial u}{\partial t}+(u\cdot\nabla)u-\nu\Delta u+\nabla p=f,\text{ in }\Omega,\\
        \nabla\cdot u = 0,\text{ in }\Omega,\\
        u=g,\text{ in } \Gamma.
    \end{split}
\end{equation}
When applying a numerical method to solve both systems of equations, \eqref{eq:THStokes} and \eqref{eq:THNS}, time has to be discretized too. However, time discretization is out of the scope of this paper (it can be found on section 4 of \cite{TH}).\enter
First of all, let $T_h=\{K_1,\dots,K_m\}$ be a \textit{discretization} of the domain $\Omega$. That is, $\Omega=K_1\cup\dots\cup K_m$, and $K_1,\dots,K_m$ don't overlap between them, and no vertex of one of them lies on the edge of another (check triangulation on Section \ref{sec:Lagrange}). If $\Omega$ is in 2D, $K_i$ is a quadrilateral, and if $\Omega$ is in 3D, $K_i$ is a a hexahedron. Then, define the following finite element function spaces on $T_h$, where $d\in\{2,3\}$ is the dimension of $\Omega$ \cite{TH}.
\begin{equation*}
    \begin{split}
        &U_h^k = \{v\in (H^1(\Omega))^d\ :\ v(K)\in(\mathcal{Q}_k(K))^d\text{ for all }K\in T_h\},\\
        &P_h^k = \{s\in H^1(\Omega)\ :\ s(K)\in\mathcal{Q}_k(K)\text{ for all }K\in T_h\},
    \end{split}
\end{equation*}
where $\mathcal{Q}_k(K)$ is the set of all polynomials whose degree on each of their variables is less or equal than $k$, with domain $K$. For example, if $K$ is two dimensional, $\mathcal{Q}_2(K)=\{a_0+a_1x+a_2y+a_3x^2+a_4y^2+a_5xy+a_6x^2y+a_7xy^2+a_8x^2y^2:a_i\in\mathbb{R},\ i=1,\dots,8\}$. Notice that the total degree of $a_8x^2y^2$ is $4$, but the degree on each variable ($x$ or $y$) is just $2$, as desired.\enter
As done on previous sections, multiply the equations of \eqref{eq:THStokes} by some test function $v\in U_h^k$ and $s\in P_h^k$, respectively, to obtain
\begin{equation}\label{THDeduction1}
    \begin{split}
        &\int_\Omega\frac{\partial u}{\partial t}v-\nu\int_\Omega\Delta u\ v+\int_\Omega\nabla p\cdot v=\int_\Omega fv,\\
        &\int_\Omega\left(\nabla\cdot u\right)s = 0.
    \end{split}
\end{equation}
Now, applying Green's formula \eqref{Green} with homogeneous boundary condition ($u=g=0$ in $\Gamma$), as done on Section \ref{sec:Lagrange}, the term $-\nu\int_\Omega\Delta u\ v$ becomes $\nu\int_\Omega\nabla u\cdot\nabla v$. Also, it is usual to multiply the second equation by $-1$, therefore, \eqref{THDeduction1} becomes the finite element formulation \eqref{THDeduction2} found on \cite{TH}. The idea is to find $(u,p)\in(U_h^k,P_h^k)$, for all $(v,s)\in(U_h^k,P_h^k)$, such that
\begin{equation}\label{THDeduction2}
    \begin{split}
        &\int_\Omega\frac{\partial u}{\partial t}v-\nu\int_\Omega\nabla u\cdot\nabla v+\int_\Omega\nabla p\cdot v=\int_\Omega fv,\\
        &-\int_\Omega\left(\nabla\cdot u\right)s = 0.
    \end{split}
\end{equation}
Let $\{\phi_i:i=1,\dots,n\}$ be a basis for $U_h^k$ and $\{\psi_j:j=1,\dots,m\}$ be a basis for $P_h^k$. Therefore, $$u(x,t)=\sum_{i=1}^nu_i(t)\phi_i(x)$$ and $$p(x,t)=\sum_{j=1}^mp_j(t)\psi_j(x)$$ where $u_i$ and $p_j$ are functions depending only on $t$. Replacing these representations on \eqref{THDeduction2} and noticing that $(\phi_i,\psi_j)\in(U_h^k,P_h^k)$, we get the system of equations \eqref{THDeduction3}, with $I=1,\dots,n$ and $J=1,\dots,m$.
\begin{equation}\label{THDeduction3}
    \begin{split}
        &\int_\Omega\frac{\partial \left(\sum_{i=1}^nu_i\phi_i\right)}{\partial t}\phi_I-\nu\int_\Omega\nabla \left(\sum_{i=1}^nu_i\phi_i\right)\cdot\nabla \phi_I\\+&\int_\Omega\nabla\left(\sum_{j=1}^mp_j\psi_j\right)\cdot\phi_I=\int_\Omega f\phi_I,\\
        &-\int_\Omega\left(\nabla\cdot \left(\sum_{i=1}^nu_i\phi_i\right)\right)\psi_J = 0.
    \end{split}
\end{equation}
After rearranging scalars, using properties of the dot product and the operator $\nabla$, and noting that $\phi_i$ does not depend on $t$, \eqref{THDeduction3} can be formulated as \eqref{THDeduction4}.
\begin{equation}\label{THDeduction4}
    \begin{split}
        &\int_\Omega\left(\sum_{i=1}^n\phi_i\frac{\partial u_i}{\partial t}\right)\phi_I-\nu\int_\Omega\left(\sum_{i=1}^nu_i\nabla \phi_i\right)\cdot\nabla \phi_I\\+&\int_\Omega\left(\sum_{j=1}^mp_j\nabla\psi_j\right)\cdot\phi_I=\int_\Omega f\phi_I,\ \ I=1,\dots,n\\
        &-\int_\Omega\left(\sum_{i=1}^nu_i\nabla\cdot \phi_i\right)\psi_J = 0,\ J=1,\dots,m.
    \end{split}
\end{equation}
Finally, \eqref{THDeduction4} can be formulated as \eqref{THDeduction5} after swapping integrals with summations and rearranging integration scalars.
\begin{equation}\label{THDeduction5}
    \begin{split}
        &\sum_{i=1}^n\frac{\partial u_i}{\partial t}\int_\Omega\phi_i\phi_I-\sum_{i=1}^nu_i\int_\Omega\nu\nabla \phi_i\cdot\nabla \phi_I\\+&\sum_{j=1}^mp_j\int_\Omega\phi_I\cdot\nabla\psi_j=\int_\Omega f\phi_I,\ \ I=1,\dots,n\\
        &-\sum_{i=1}^nu_i\int_\Omega\psi_J\nabla\cdot \phi_i = 0,\ J=1,\dots,m.
    \end{split}
\end{equation}
As before, the problem can be reduced to the matrix system \eqref{THDeduction6}, which is the semi-discrete Stokes problem \cite{TH}.
\begin{equation}\label{THDeduction6}
\begin{split}
    M\dot{u}+Lu+Gp=f,\\
    -Du=0,
\end{split}
\end{equation}
where $M$ and $L$ are $n\times n$ matrices, $G$ is a $n\times m$ matrix, $D$ is a $m\times n$ matrix, $u$ and $f$ are $n$-dimensional vectors, $p$ is a $m$-dimensional vector and $\dot{u}$ is the notation used for the partial derivate of $u$ with respect to time $t$. The entries of these arrays are $M[i,j]=\int_\Omega\phi_i\phi_j$, $L[i,j]=\int_\Omega\nu\nabla\phi_i\cdot\nabla\phi_j$, $G[i,j]=\int_\Omega\phi_i\cdot\nabla\psi_j$, $D[i,j]=\int_\Omega\psi_i\nabla\cdot\phi_j$, $f[i]=\int_\Omega f\phi_i$, $p[i]=p_i$, $u[i]=u_i$ and $\dot{u}[i]=\frac{\partial u_i}{\partial t}$.\enter
As mentioned on \cite{TH}, for the steady Stokes problem, $\dot{u}=0$ is taken. In such case, \eqref{THDeduction6} becomes the linear matrix system \eqref{THDeduction7}.
\begin{equation}\label{THDeduction7}
    \begin{pmatrix}
    L & G\\
    -D & 0
    \end{pmatrix}\begin{pmatrix}
    u\\
    p
    \end{pmatrix}=\begin{pmatrix}
    f\\
    0
    \end{pmatrix}
\end{equation}
Furthermore, for the Navier-Stokes equations \eqref{eq:THNS}, the semi-discrete formulation is \cite{TH}:
\begin{equation}\label{THDeduction8}
\begin{split}
    M\dot{u}+Lu+\mathcal{N}(u)+Gp=f,\\
    -Du=0,
\end{split}
\end{equation}
where $$\mathcal{N}(u)[i]=\bigintss_\Omega\begin{pmatrix}
u_1 & \dots & u_n
\end{pmatrix}\begin{pmatrix}
(\phi_1\cdot\nabla)\phi_1 & \dots & (\phi_1\cdot\nabla)\phi_n\\
\vdots & \ddots & \vdots\\
(\phi_n\cdot\nabla)\phi_1 & \dots & (\phi_n\cdot\nabla)\phi_n\\
\end{pmatrix}\begin{pmatrix}
u_1\\
\vdots\\
u_n
\end{pmatrix}\phi_i$$
is the discretized nonlinear vector-convection term, of size $n\times n$.\enter
Finally, recall that the spaces $U_h^k$ and $P_h^k$ have a given order $k$. As mentioned on \cite{TH}, Taylor-Hood finite element space is the tuple $(U_h^k,P_{k-1})$, which is used to solve steady and unsteady Stokes problem (convergence is optimal and stable for $k\geq2$). And, for Navier-Stokes equations, the finite element space used is $(U_h^k,P_h^k)$, called $P_NP_N$ space.\enter
According to MFEM documentation \cite{MFEM}, the implementation for the solution of Navier-Stokes equations is done following \cite{TH}, which is the theory presented on this section.
\subsection{MFEM Library}\label{sec:MFEM}
In this manuscript,  we worked with MFEM's \textit{Example\#1} and \textit{Example\#5} which can be found on \cite{MFEM}. Example\#1 uses standard Lagrange finite elements and Example\#5 uses Raviart-Thomas mixed finite elements. Further, in Section \ref{sec:EquivProblem}, we find the parameters so that both problems are equivalent and then (Section \ref{sec:Results}), we compare the solutions.\enter
We finally mention that for a fair comparison between Lagrange and mixed finite element's approximation, Lagrange shape functions of order $k-1$ will be compared to the corresponding (mixed) approximation obtained by using $RT_k(K)$.\enter
Afterwards, we worked with MFEM's miniapp for solving Navier-Stokes equations, which corresponds to the experiments done on Section \ref{sec:NSExp}. On Section \ref{sec:NS2D} we worked in a 2-dimensional domain, and on Section \ref{sec:NS3D} we worked in a 3-dimensional domain.
\subsubsection{Information about the library}\label{sec:MFEMinfo}
According to it's official site \cite{MFEM}, MFEM is a free, lightweight, scalable C++ library for finite element methods that can work with arbitrary high-order finite element meshes and spaces.\\
\newline
MFEM has a serial and a parallel version. The serial version is the one recommended for beginners, and is used in Section \ref{sec:Vs}. On the other hand, the parallel version provides more computational power and enables the use of some MFEM \texttt{mini-apps}, like the Navier-Stokes mini app, which is used in Section \ref{sec:NSExp}. \\
\newline
Moreover, the Modular Finite Element Method (MFEM) library is developed by the MFEM Team at the Center for Applied Scientific Computing (CASC), located in the Lawrence Livermore National Laboratory (LLNL), under the BSD licence. However, as it is open source, the public repository can be found at \texttt{github.com/mfem} in order for anyone to contribute.
\enter
Also, since 2018, a wrapper for Python (PyMFEM) is being developed in order to use MFEM library among with Python code, which demonstrates the wide applicability that the library can achieve. And, in 2021, the first community workshop was hosted by the MFEM Team, which encourages the use of the library and enlarges the community of MFEM users.
\enter
Finally, take into account that the use of the library requires a good manage of C++ code, which is a programming language that's harder to use compared to other languages, such as Python. This understanding of C++ code is important because some parts of the library are not well documented yet, and, by checking the source code, the user may find a way of implementing what is required.
\subsubsection{Overview}\label{sec:MFEMoverview}
The main classes (with a brief and superficial explanation of them) that we are going to use in the code are:
\begin{itemize}
    \item Mesh: domain with the partition.
    \item FiniteElementSpace: space of functions defined on the finite element mesh.
    \item GridFunction: mesh with values (solutions).
    \item $\_$Coefficient: values of GridFunctions or constants.
    \item LinearForm: maps an input function to a vector for the rhs.
    \item BilinearForm: used to create a global sparse finite element matrix for the lhs.
    \item $\_$Vector: vector.
    \item $\_$Solver: algorithm for solution calculation.
    \item $\_$Integrator: evaluates the bilinear form on element's level.
    \item NavierSolver: class associated to the navier mini-app, which is used to solve Navier-Stokes equations.
\end{itemize}
The ones that have $\_$ are various classes whose name ends up the same and work similarly.\\
\newline
\textit{\underline{Note:}} \\\textbf{lhs}: left hand side of the linear system.\\
\textbf{rhs}: right hand side of the linear system.

\subsubsection{Code structure}\label{sec:CodeStruct}
A MFEM general code has the following steps (directly related classes with the step are written):
\begin{enumerate}
    \item Receive an input file (.msh) with the mesh and establish the order for the finite element spaces.
    \item Create a mesh object, get the dimension, and refine the mesh (refinement is optional). \texttt{Mesh}
    \item Define the finite element spaces required. \texttt{FiniteElementSpace}
    \item Define the coefficients, functions, and boundary conditions of the problem. \texttt{\textit{X}Coefficient}
    \item Define the LinearForm for the rhs and assemble it. \texttt{LinearForm, \textit{X}Integrator}
    \item Define the BilinearForm for the lhs and assemble it. \texttt{BilinearForm, \textit{X}Integrator}
    \item Solve the linear system. \texttt{\textit{X}Solver, \textit{X}Vector}
    \item Recover solution. \texttt{GridFunction}
    \item Show solution with a finite element visualization tool like \textbf{GLVis} \cite{Glvis} (optional).
\end{enumerate}
And, for the general code structure of the navier mini-app, we have the following steps:
\begin{enumerate}
    \item Receive an input file (.msh) with the mesh, create a parallel mesh object and refine the mesh (refinement is optional). \texttt{ParMesh}
    \item Create the flow solver by stating the order of the finite element spaces and the parameter for kinematic viscosity $\nu$. \texttt{NavierSolver}
    \item Establish the initial condition, the boundary conditions and the time step $dt$.
    \item Iterate through steps in time with the NavierSolver object and save the solution for each iteration in a parallel GridFunction. \texttt{ParGridFunction}
    \item Show the solution with a finite element visualization tool like \textbf{ParaView} \cite{paraview} (optional).
\end{enumerate}
Notice that the Mesh and GridFunction classes used in the mini-app are for the parallel version of MFEM. The reason for this, is that the navier mini-app is available for the parallel version of MFEM only. Also, the mini-app is coded in such way that the code is simple (see Appendix \ref{sec:appC}).
\newpage

\section{Lagrange vs. Raviart-Thomas finite elements}\label{sec:Vs}
In this section, we take examples 1 and 5 from \cite{MFEM}, define their problem parameters in such way that they're equivalent, create a code that implements both of them at the same time and compares both solutions ($L_2$ norm), run the code with different orders, and analyse the results.\\
\newline
Some considerations to have into account for a fair comparison are that, the order for the Mixed method should be 1 less than the order for Lagrange method, because, with this, both shape functions would have the same degree. Also, we will compare pressures and velocities with respect to the order of the shape functions and the size of the mesh ($h$ parameter). Furthermore, for the problem, the exact solution is known, so, we will use it for comparison. And, the maximum order and refinement level to be tested is determined by our computational capacity (as long as solvers converge fast). 

\subsection{Problem}\label{sec:EquivProblem}
As mentioned before, we have to find the parameters for example 1 and 5 from \cite{MFEM}, in such way that both problems are equivalent. This step is important because example \# 1 is solved using Lagrange finite elements, while example \# 5 is solved using mixed finite elements. Therefore, in order to make the comparison, the problem must be the same for both methods.\enter
Example\#1 \cite{MFEM}: Compute $p$ such that
\begin{equation} \label{Example1}
\begin{split}
        -\Delta &p=1\text{ in } \Omega\\
        &p=0\text{ in }\Gamma.
\end{split}
\end{equation}
Example\#5 \cite{MFEM}: Compute $p$ and $\mathbf{u}$ such that 
\begin{equation}  \label{Example5}
\begin{split}
        &k\mathbf{u}+\nabla p=f\text{ in } \Omega\\
        &-\mbox{div}(\mathbf{u})=g\text{ in } \Omega\\
        &-p=p_0\text{ in }\Gamma.
\end{split}
\end{equation}
From the first equation of \eqref{Example5},
\begin{equation}\label{Deduction1}
    \mathbf{u}=\frac{f-\nabla p}{k}.
\end{equation}
Then, replacing \eqref{Deduction1} on the second equation of \eqref{Example5},
\begin{equation}\label{Deduction2}
    -\mbox{div}\left(\frac{f-\nabla p}{k}\right)=g.
\end{equation}
If we set $k=1;\ f=0\ and\ g=-1$ in \eqref{Deduction2}, we get
\begin{equation}\label{Deduction3}
    -\Delta p=1,
\end{equation}
which is the first equation of \eqref{Example1}. \\
\newline
So, setting \textbf{($*$)} $p_0=0,\ k=1;\ f=0\ and\ g=-1$ in \eqref{Example5}, we get,
\begin{equation}
\begin{split}
        &\mathbf{u}+\nabla p=0\text{ in } \Omega\\
        &-\mbox{div}(\mathbf{u})=-1\text{ in } \Omega\\
        &-p=0\text{ in }\Gamma.
\end{split}
\end{equation}
Notice that from the first equation we get that $\mathbf{u}=-\nabla p$. This is important because in problem \eqref{Example1} we don't get the solution for $\mathbf{u}$ from the method, so, we will have to find it from $p$'s derivatives.\\
\newline
In the code, we will set the value of the parameters in the way shown here, so that both problems are the same. As seen in \eqref{Deduction1}-\eqref{Deduction3}, problem \eqref{Example5} is equivalent to problem \eqref{Example1} with the values assigned for coefficients and functions at \textbf{($*$)}.
\subsection{Code}\label{VsCode}
The first part of the code follows the structure mentioned in Section \ref{sec:CodeStruct}, but implemented for two methods at the same time (and with some extra lines for comparison purposes). Also, when defining boundary conditions, the \textit{essential} one is established different from the \textit{natural} one. And, after getting all the solutions, there's a second part of the code where solutions are compared between them and with the exact one. \\
\newline
\textit{Note:} \\The complete code with explanations can be found on the \texttt{Appendix A}.\\
\newline
However, before taking a look into it, the reader may have the following into account. The following table shows the convention used for important variable names along the code:
\begin{table}[h!]
\begin{tabular}{|c|c|}
\hline
\textbf{Variable Name} & \textbf{Object}                               \\ \hline
X\_space               & Finite element space X                        \\ \hline
X\_mixed               & Variable assigned to a mixed method related object    \\ \hline
u                      & Velocity solution                                     \\ \hline
p                      & Pressure solution                                    \\ \hline
X\_ex                  & Variable assigned to an exact solution object \\ \hline
\end{tabular}
\end{table}

\subsection{Tests}\label{sec:VsTests}
The tests of the two methods presented previously, were run on MFEM library on the domain shown on Figure \ref{fig:Domain}.
\begin{figure}[h!]
    \centering
    \includegraphics[scale=0.5]{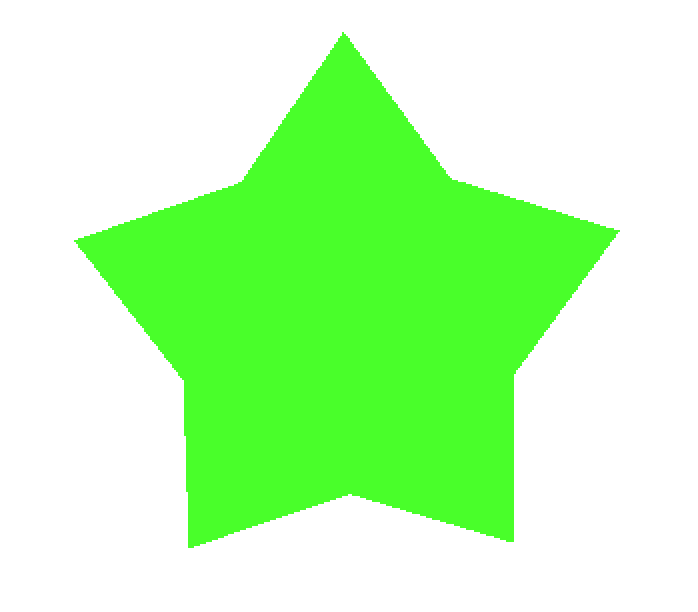}
    \caption{Illustration of the star domain used for the numerical tests.\\ Visualization: \cite{Glvis}.}
    \label{fig:Domain}
\end{figure}\\
Each run test is determined by the \textit{order} of the Lagrange shape functions and the \textit{h} parameter of the mesh. Remember that mixed shape functions have order equal to $\textit{order}-1$. The parameter \textit{order} is changed directly from the command line, while the parameter \textit{h} is changed via the number of times that the mesh is refined ($h=h(\#refinements)$). As we refine the mesh more times, finite elements of the partition decrease their size, and so, the parameter $h$ decreases.\\
\newline
Tests were run with: $order=1,\dots,N$ and $refinements=0,\dots,M$, where $N,M$ depend on the computation capacity. The star domain was partitioned using quads (instead of triangles), and such partition is shown on Figure \ref{fig:Partition}.
\begin{figure}[h!]
    \centering
    \includegraphics[scale=0.5]{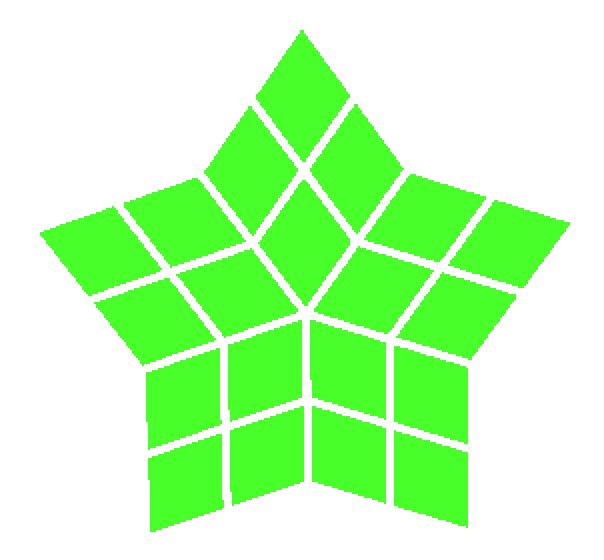}
    \caption{Initial mesh used for numerical tests (no refinements). \\Visualization: \cite{Glvis}.}
    \label{fig:Partition}
\end{figure}
\newpage
Results on Section \ref{sec:Results} are presented in graphs. However, all the exact values that were computed can be found in the \texttt{Appendix B}.
\subsection{Results}\label{sec:Results}
In Figure \ref{fig:Solution} we show the computed solution
 when running the code with $\textit{order}=2$ and $\#Refinements=3$. 
We use the visualization tool \cite{Glvis}. We mention that, 
at the scale of the plot, Lagrange and Mixed solutions look the same. 
\begin{figure}[h!]
    \centering
    \includegraphics[scale=0.25]{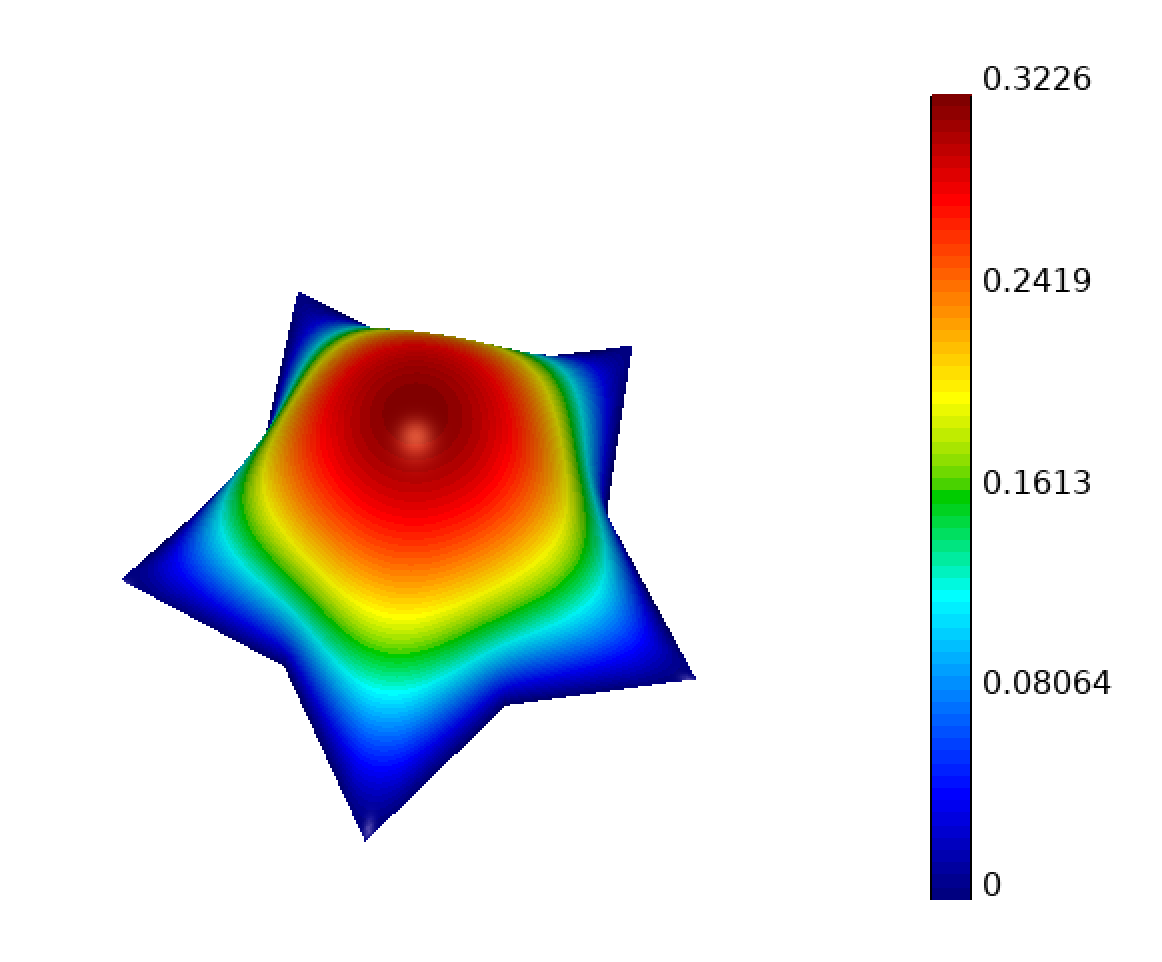}
    \includegraphics[scale=0.25]{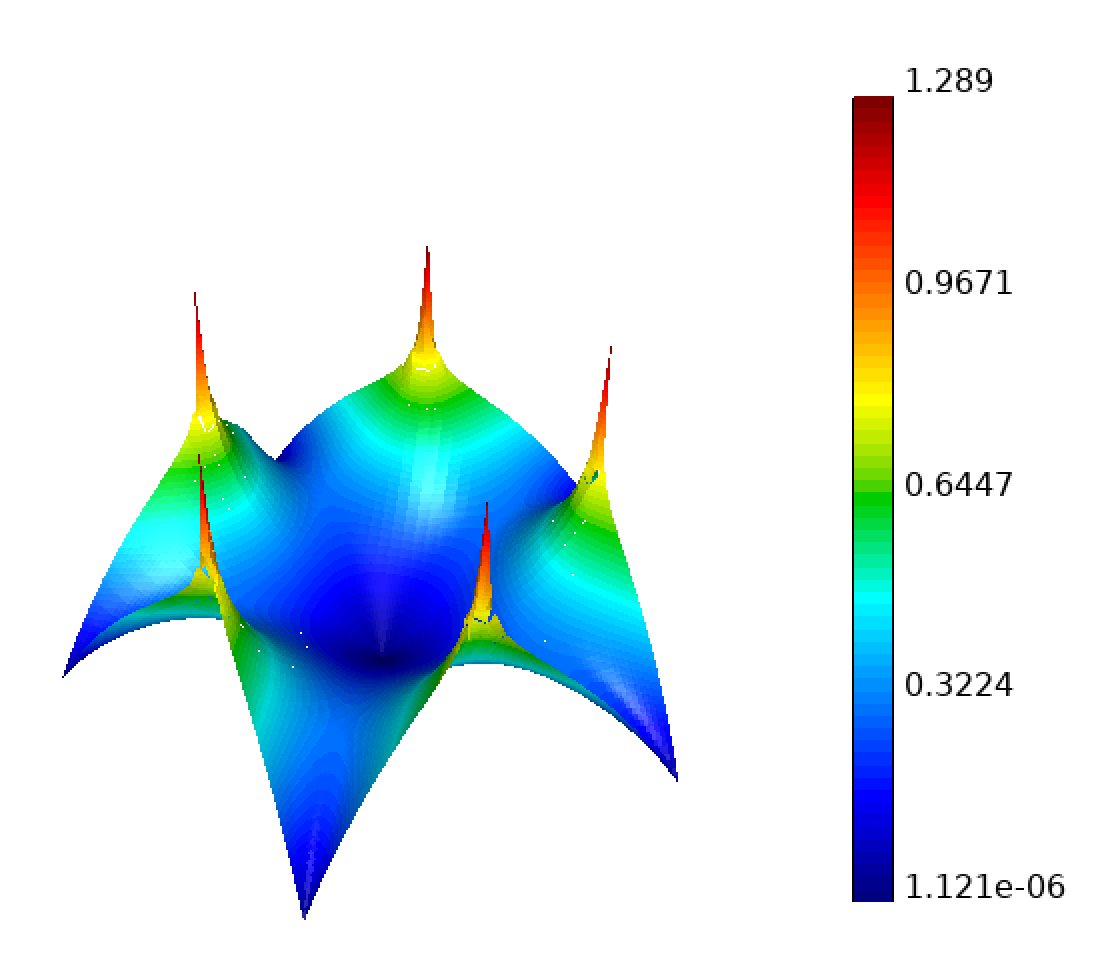}
    \caption{Illustration of computed pressure and velocities. GLVis (\cite{Glvis}) is used for this visualization. Pressure (left) and $L^2$ norm of the vector velocity (right). Visualization: \cite{Glvis}.}
    \label{fig:Solution}
\end{figure}\\
In the following results, if $u=(u_x,u_y)$ is the solution obtained by the mixed or Lagrange finite element method and $u_{ex}=(u_{x_{ex}},u_{y_{ex}})$ is the exact solution for the problem, then,
\begin{equation*}
    U_{error} = \frac{\sqrt{\left(||u_x-u_{x_{ex}}||_{L^2}\right)^2+\left(||u_y-u_{y_{ex}}||_{L^2}\right)^2}}{||u_{ex}||_{L^2}}.
\end{equation*}
\begin{figure}[h!]
    \centering
    \includegraphics[scale = 0.45]{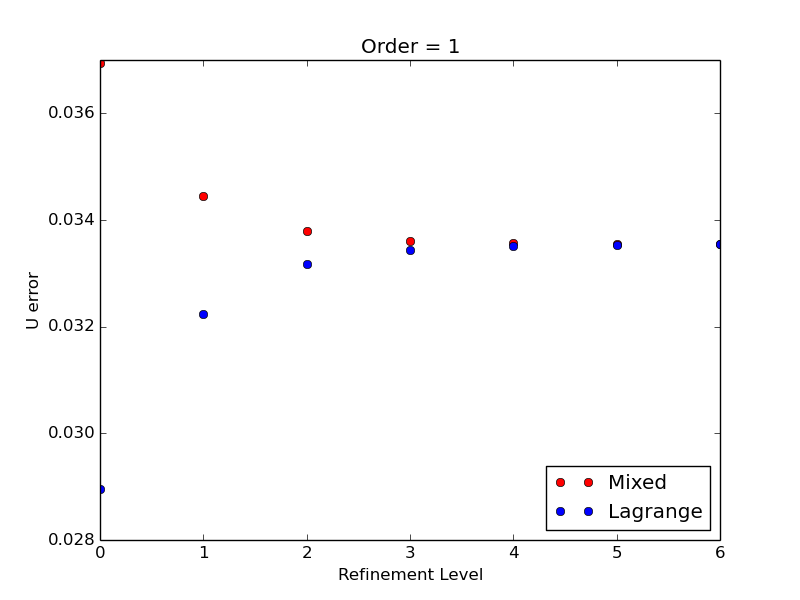}
    \caption{Variation of error with respect to the refinement level for the approximation of the solution of 
    problem \eqref{Example1} with Lagrangian finite elements of order 1 and problem \eqref{Example5} with mixed Raviart-Thomas finite elements of order 0.}
\end{figure}
\begin{figure}[h!]
    \centering
    \includegraphics[scale = 0.45]{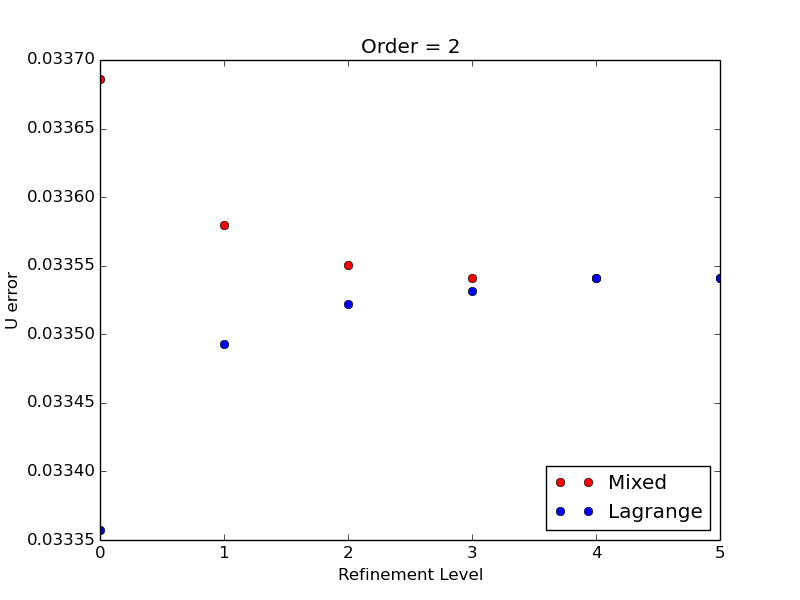}
    \caption{Variation of error with respect to the refinement level for the approximation of the solution of 
    problem \eqref{Example1} with Lagrangian finite elements of order 2 and problem \eqref{Example5} with mixed Raviart-Thomas finite elements of order 1.}
\end{figure}
\begin{figure}[h!]
    \centering
    \includegraphics[scale = 0.45]{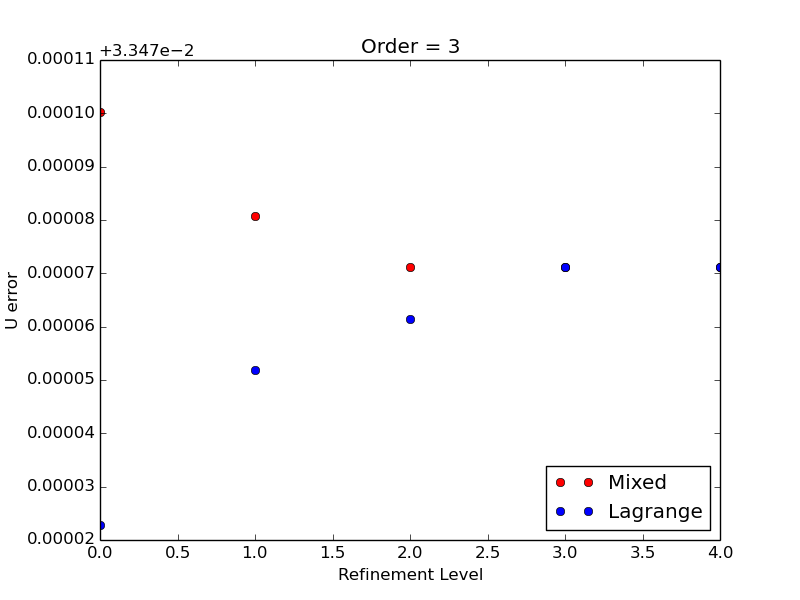}
    \caption{Variation of error with respect to the refinement level for the approximation of the solution of 
    problem \eqref{Example1} with Lagrangian finite elements of order 3 and problem \eqref{Example5} with mixed Raviart-Thomas finite elements of order 2.}
\end{figure}
\begin{figure}[h!]
    \centering
    \includegraphics[scale = 0.45]{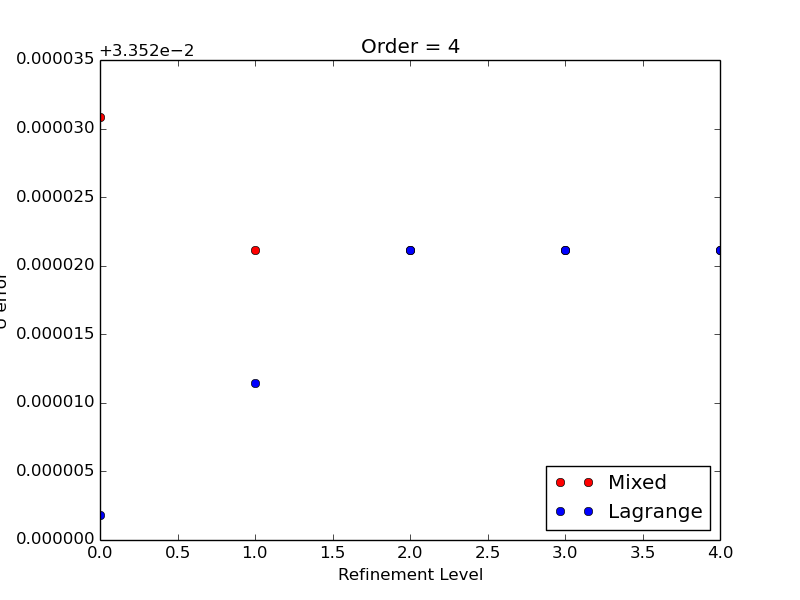}
    \caption{Variation of error with respect to the refinement level for the approximation of the solution of 
    problem \eqref{Example1} with Lagrangian finite elements of order 4 and problem \eqref{Example5} with mixed Raviart-Thomas finite elements of order 3.}
\end{figure}
\newpage
\subsection{Analysis}\label{sec:Analysis}
In this section we comment and analyze the results in tables presented on the Appendix B.\\
\newline
To understand the information presented, take into account that the exact solution would have value $0$ in \textit{X err}. Also, if the two solutions obtained (Lagrange and Mixed) are exactly the same, the value in \textit{P comp} and \textit{U comp} would be $0$. And, lower values of $h$ mean more mesh refinements, ie, smaller partition elements.\\
\newline
As expected, computational time increases as order and refinements increase. Here are the most relevant observations that can be obtained after analysing the data corresponding to \textbf{absolute errors}.\\
\newline
For fixed order, absolute errors have little variation when reducing $h$ (max variation is $4.722$e$-03$ in $Uerr$ order 1); $Perr$ increases as $h$ decreases, while $Pmx\ err$ decreases as $h$ decreases; and, $Uerr$ increases as $h$ decreases, while $Umx\ err$ decreases as $h$ decreases. \\
\newline
Absolute errors variation (respect to refinement) is lower when order is higher. For example; in order 2, $Perr$ is the same for each $h$ (up to three decimal places); while in order 6, $Perr$ is the same for each $h$ (up to five decimal places).\\
\newline
For fixed $h$, absolute errors remain almost constant between orders. Moreover, $Perr$ (absolute error obtained for pressure with Lagrange) is always lower than $Pmx\ err$ (absolute error obtained for pressure with mixed) and $Uerr$ (absolute error obtained for velocity with Lagrange) is always lower than $Umx\ err$ (absolute error obtained for velocity with mixed).\\
\newline
As order increases, pressure and velocity absolute errors tend to be the same. In order 10, the difference between $Perr$ and $Pmx\ err$ is $0.000001$ and the difference between $Uerr$ and $Umx\ err$ is $<0.0000009$.\\
\newline
However, notice that in all the cases, the absolute error was higher than $1$. In $L_2$ norm, this value is pretty little, and shows that we are only getting approximations of the exact solution.\enter
Now, the most relevant observations that can be obtained after analysing the data corresponding to \textbf{comparison errors}. First of all, comparison error tends to $0$; and comparison errors $Ucomp$ and $Pcomp$ decrease as $h$ decreases.\\
\newline
For a fixed order, comparison error can be similar to a higher order comparison error, as long as enough refinements are made. Moreover, when order increases, comparison errors are lower for fixed $h$.\enter
Pressure comparison error lowers faster than velocity comparison error. Maximum comparison errors were found at order 1 with no refinements, where $Pcomp\approx7.5$e$-02$ and $Ucomp\approx3.7$e$-02$, and minimum comparison errors were found at order 10 with 1 refinement (higher refinement level computed for order 10), where $Pcomp\approx5.1$e$-06$ and $Ucomp\approx9.8$e$-04$. It can be seen that $Pcomp$ improved in almost four decimal places while $Ucomp$ improved in just 2.

\subsection{Some other examples}\label{sec:MFEMexamples}
In this section we show three of the examples that MFEM library provides \cite{MFEM}. We only show the problem, a brief verification of the exact solution and the solution obtained using MFEM, without going into details of any type. The purpose is to show the wide variety of applications that MFEM library can have and let the reader familiarize with the visualization of some finite element method solutions. \enter
\texttt{Example 1}\\
This example is Example \# 3 of \cite{MFEM} and consists on solving the second order definite Maxwell equation $$\nabla\times\nabla\times E+E=f$$with Dirichlet boundary condition. In the example, the value for $f$ is given by $$f\begin{pmatrix}x\\y\\z\end{pmatrix}=\begin{pmatrix}\left(1+\pi^2\right)\sin(\pi y)\\\left(1+\pi^2\right)\sin(\pi z)\\\left(1+\pi^2\right)\sin(\pi x)\end{pmatrix}.$$
The exact solution for $E$ is $$E\begin{pmatrix}x\\y\\z\end{pmatrix}=\begin{pmatrix}\sin(\pi y)\\\sin(\pi z)\\\sin(\pi x)\end{pmatrix},$$ and can be verified by computing
\begin{equation*}\label{eq:verification1}
\begin{split}
    &\nabla\times\nabla\times E+E\\
    =&\vectortres{\frac{\partial}{\partial x}}{\frac{\partial}{\partial y}}{\frac{\partial}{\partial z}}\times\vectortres{\frac{\partial}{\partial x}}{\frac{\partial}{\partial y}}{\frac{\partial}{\partial z}}\times\vectortres{\sin(\pi x)}{\sin(\pi z)}{\sin(\pi y)}+\vectortres{\sin(\pi y)}{\sin(\pi z)}{\sin(\pi x)}\\
    =&\vectortres{\frac{\partial}{\partial x}}{\frac{\partial}{\partial y}}{\frac{\partial}{\partial z}}\times\vectortres{-\pi\cos(\pi z)}{-\pi\cos(\pi x)}{-\pi\cos(\pi y)}+\vectortres{\sin(\pi y)}{\sin(\pi z)}{\sin(\pi x)}\\
    =&\vectortres{\pi^2\sin(\pi y)}{\pi^2\sin(\pi z)}{\pi^2\sin(\pi x)}+\vectortres{\sin(\pi y)}{\sin(\pi z)}{\sin(\pi x)} = f.
\end{split}
\end{equation*}
The solution for $E$, $E_h$, computed with MFEM library is presented on Figure \ref{fig:example1}. The error for the approximation is $$\left|\left|E_h-E\right|\right|_{L^2}=0.39154.$$
\begin{figure}[h!]
    \centering
    \includegraphics[scale=0.3]{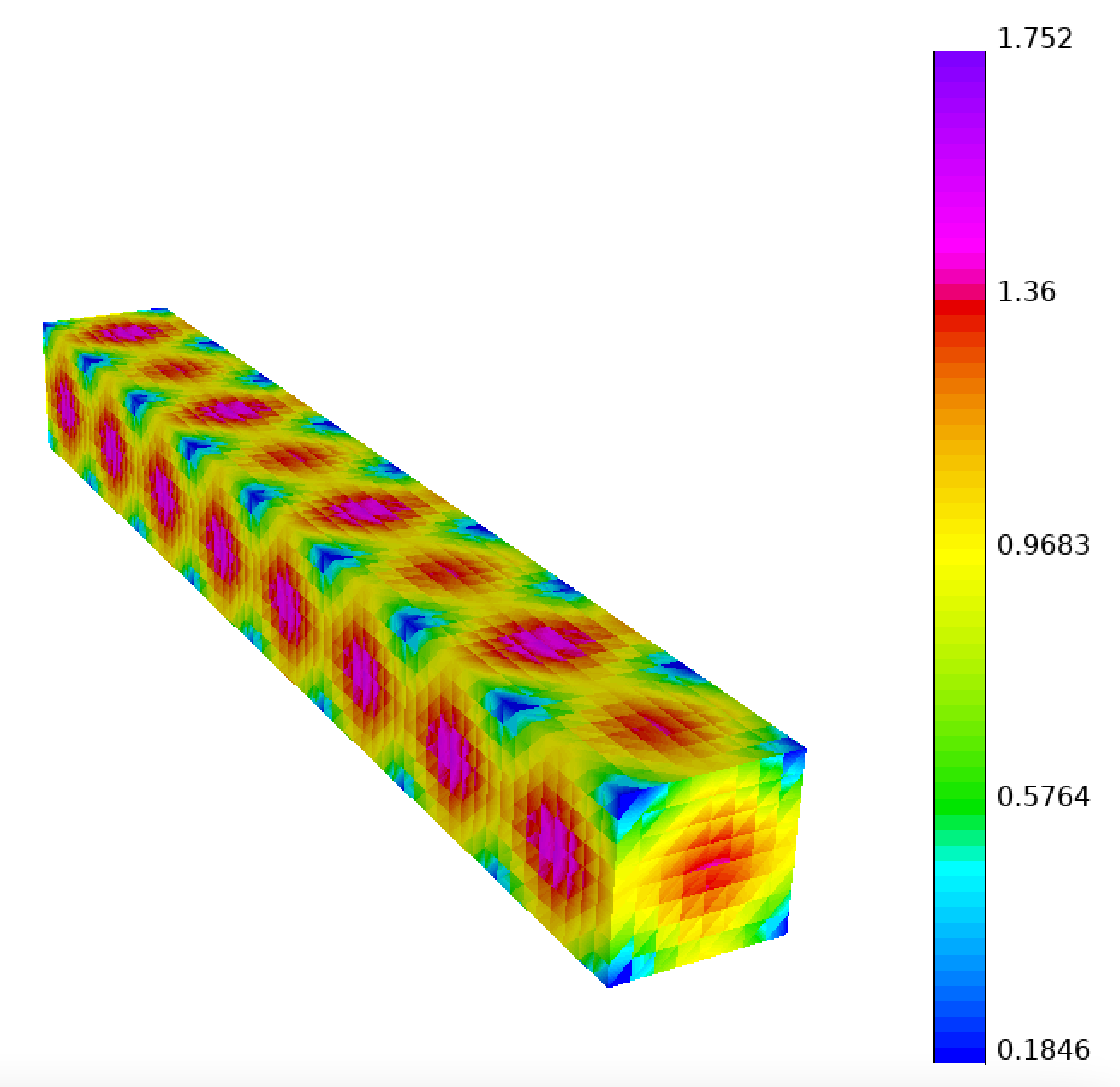}
    \caption{Visualization of the norm of the solution for the electromagnetic diffusion problem corresponding to the second order definite Maxwell equation. Solution for $E$, $E_h$, obtained using MFEM library. Visualization: \cite{Glvis}.}
    \label{fig:example1}
\end{figure}\\
\texttt{Example 2}\\
This example is Example \# 4 of \cite{MFEM} and consists on solving the diffusion problem corresponding to the second order definite equation $$-\nabla(\alpha\mbox{div}(F))+\beta F=f$$ with Dirichlet boundary condition, and, with parameters $\alpha=1$ and $\beta=3$. In the example, the value for $f$ is given by $$f\begin{pmatrix}x\\y\end{pmatrix}=\vectordos{(3+2\pi^2)\cos(\pi x)\sin(\pi y)}{(3+2\pi^2)\cos(\pi y)\sin(\pi x)}.$$
The exact solution for $F$ is $$F\begin{pmatrix}x\\y\end{pmatrix}=\vectordos{\cos(\pi x)\sin(\pi y)}{\cos(\pi y)\sin(\pi x)},$$ and can be verified by computing
\begin{equation*}\label{eq:verification2}
\begin{split}
    &-\nabla(\alpha\mbox{div}(F))+\beta F\\
    =&-\nabla\left(\mbox{div}\vectordos{\cos(\pi x)\sin(\pi y)}{\cos(\pi y)\sin(\pi x)}\right)+ 3\vectordos{\cos(\pi x)\sin(\pi y)}{\cos(\pi y)\sin(\pi x)}\\
    =&-\nabla\left(-2\pi\sin(\pi x)\sin(\pi y)\right)+ 3\vectordos{\cos(\pi x)\sin(\pi y)}{\cos(\pi y)\sin(\pi x)}\\
    =&-\vectordos{-2\pi^2\cos(\pi x)\sin(\pi y)}{-2\pi^2\sin(\pi x)\cos(\pi y)}+ 3\vectordos{\cos(\pi x)\sin(\pi y)}{\cos(\pi y)\sin(\pi x)}\\
    =&\vectordos{2\pi^2\cos(\pi x)\sin(\pi y)}{2\pi^2\cos(\pi y)\sin(\pi x)}+ \vectordos{3\cos(\pi x)\sin(\pi y)}{3\cos(\pi y)\sin(\pi x)} = f.
\end{split}
\end{equation*}
The solution for $F$, $F_h$, computed with MFEM library is presented on Figure \ref{fig:example2}. The domain is a square with a circular hole in the middle. On the visualization, the triangular elements of the mesh are shown. The error for the approximation is $$\left|\left|F_h-F\right|\right|_{L^2}=5.55372\times 10^{-6}.$$
\begin{figure}[h!]
    \centering
    \includegraphics[scale=0.3]{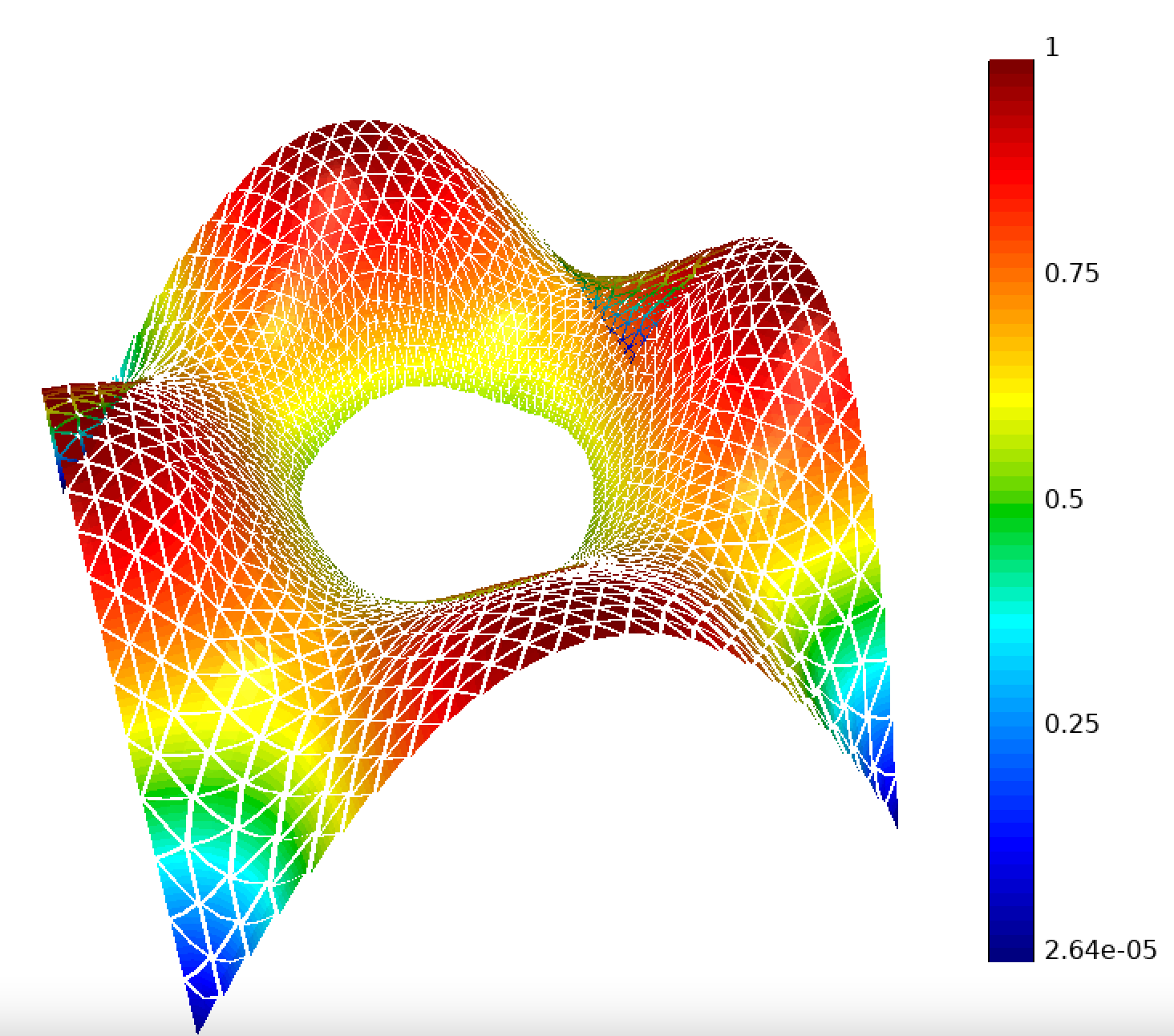}
    \caption{Visualization of the norm of the solution for the diffusion problem showing the mesh elements. Solution for $F$, $F_h$, obtained using MFEM library. Visualization: \cite{Glvis}.}
    \label{fig:example2}
\end{figure}\\
\texttt{Example 3}\\
This example is Example \# 7 of \cite{MFEM} and consists on solving the Laplace problem with mass term corresponding to the equation $$-\Delta u+u=f.$$ In the example, the value for $f$ is given by $$f\begin{pmatrix}x\\y\\z\end{pmatrix}=\frac{7xy}{x^2+y^2+z^2}.$$
The exact solution for $u$ is $$u\begin{pmatrix}x\\y\\z\end{pmatrix}=\frac{xy}{x^2+y^2+z^2},$$ and can be verified by computing
\begin{equation*}\label{eq:verification3}
\begin{split}
    &-\Delta u+u\\
    =&-\mbox{div}\left(\nabla\left(\frac{xy}{x^2+y^2+z^2}\right)\right)+\frac{xy}{x^2+y^2+z^2}\\
    =&-\mbox{div}\left(\frac{1}{(x^2+y^2+z^2)^2}\begin{pmatrix}
    y(-x^2+y^2+z^2)
    \\x(x^2-y^2+z^2)
    \\-2xyz
    \end{pmatrix}\right)+\frac{xy}{x^2+y^2+z^2}\\
    =&-\left(-\frac{6xy}{x^2+y^2+z^2}\right)+\frac{xy}{x^2+y^2+z^2}\\
    =&\frac{7xy}{x^2+y^2+z^2}=f.
\end{split}
\end{equation*}
The solution for $u$, $u_h$, computed with MFEM library is presented on Figure \ref{fig:example3}. The error for the approximation is $$\left|\left|u_h-u\right|\right|_{L^2}=0.00236119.$$
\begin{figure}[h!]
    \centering
    \includegraphics[scale=0.21]{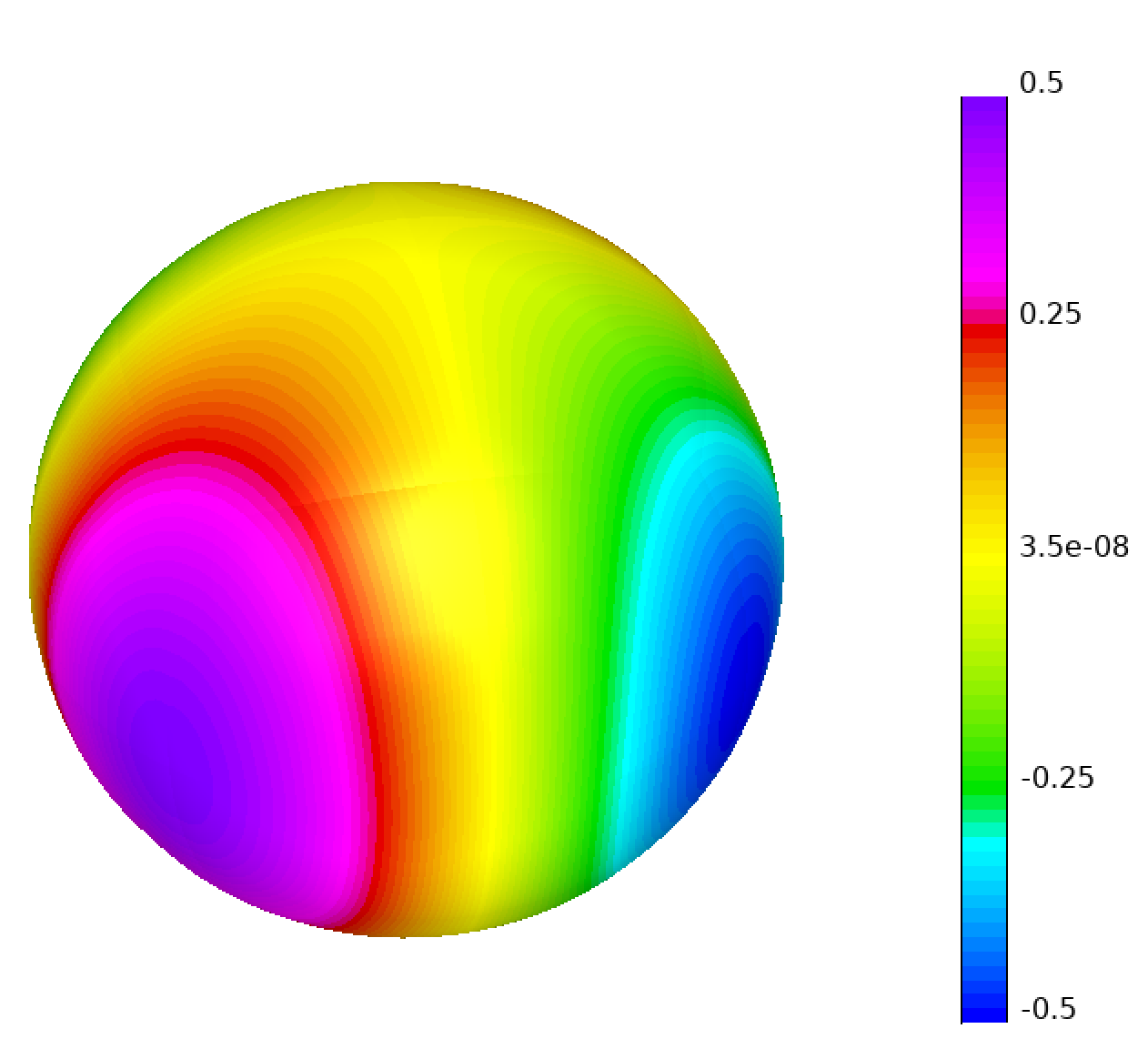}
    \caption{Visualization of the solution for the Laplace problem with mass term. Solution for $u$, $u_h$, obtained using MFEM library. Visualization: \cite{Glvis}.}
    \label{fig:example3}
\end{figure}\enter
These examples show that MFEM can solve several types of equations that can include divergence, curl, gradient and Laplacian operators. Also, the given parameter $f$ was a function of the form $\mathbb{R}^3\rightarrow\mathbb{R}^3$, $\mathbb{R}^2\rightarrow\mathbb{R}^2$ and $\mathbb{R}^3\rightarrow\mathbb{R}$, which shows that MFEM can work with scalar and vectorial functions in a 2D or 3D domain. Although the shown visualizations were norms of the solution, using another visualization tool, such as \cite{paraview}, the vectors of the solution can be seen.

\section{Numerical Experiments with NS}\label{sec:NSExp}
In this section we run some computational experiments solving Navier-Stokes equations using MFEM's Navier Stokes mini app. One of the experiments is in a 2D domain (Section \ref{sec:NS2D}) and the other one is in a 3D domain (Section \ref{sec:NS3D}). The two dimensional experiment converges to a steady state when the initial condition is the steady state with a small perturbation. In such experiment, we compare graphically the pressure solution obtained when changing the order and the refinement level. On the other hand, in the three dimensional experiment, we revise a graphical solution for a dynamical system, where turbulence is present.
\subsection{2D Experiment: Steady State}\label{sec:NS2D}
The two dimensional domain is a $2\times4$ rectangle whose vertex coordinates are $(-0.5,1.5)$, $(-0.5,-0.5)$, $(1,1.5)$ and $(1,-0.5)$, as shown on Figure \ref{fig:2Domain}.
\begin{figure}[h!]
    \centering
    \includegraphics[scale=0.5]{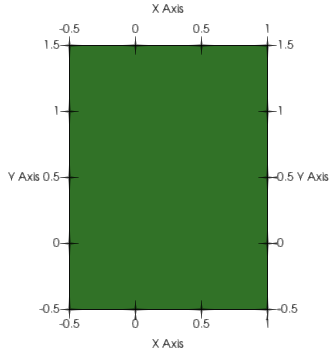}
    \caption{Two dimensional domain used in the 2D Navier-Stokes experiment. It is a $2\times4$ rectangle. Visualization: \cite{paraview}.}
    \label{fig:2Domain}
\end{figure}
\newpage
Also, the default mesh (with no refinements) is presented on figure \ref{fig:2Dmesh}.
\begin{figure}[h!]
    \centering
    \includegraphics[scale=0.4]{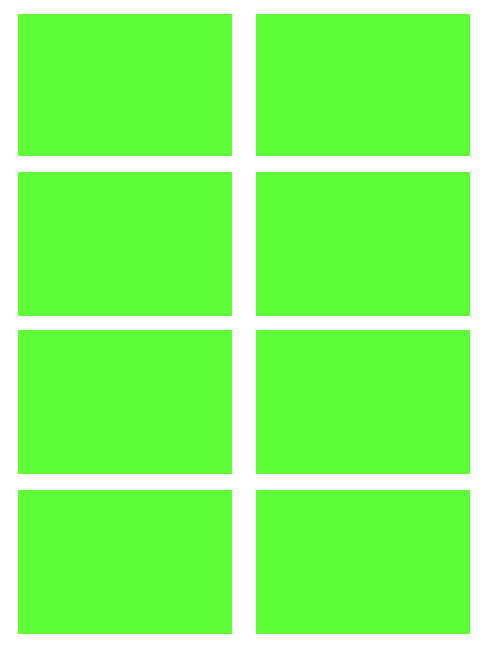}
    \caption{Two dimensional domain mesh used in the 2D Navier-Stokes experiment. Visualization: \cite{paraview}.}
    \label{fig:2Dmesh}
\end{figure}\\
\newcommand{\lamsteady}[0]{
    2\left(10-\sqrt{100+\pi^2}\right)
}
For this experiment, $Re=40$ and the velocity boundary condition was settled to be 
\begin{equation}\label{eq:boundaryvelsteady}
    u_0(x,y)=\begin{pmatrix}
1-e^{\lamsteady x} \cdot\cos\left(2\pi y\right)\\
\frac{\lamsteady}{2\pi}\cdot e^{\lamsteady x}\cdot\sin\left(2\pi y\right)
\end{pmatrix}.
\end{equation}
If the initial condition is picked to be $u_0$, then the system is already on a steady state. As we wanted the experiment to \textit{reach} the steady state, then initial condition was ensured to be
\begin{equation}\label{eq:initialvelsteady}
    u_i(x,y)=u_0(x,y)+\delta\begin{pmatrix}
(x+0.5)(x-1)(y+0.5)(y-1.5)\\
(x+0.5)(x-1)(y+0.5)(y-1.5)
\end{pmatrix},
\end{equation}
where the $\delta$ parameter in the initial condition gives the magnitude of the perturbation from the steady state. The parameter was picked to be $\delta=0.001$. Notice that the term after $\delta$ vanishes in the boundary of the domain.\enter
The expected result in pressure (computed with order 6 and 5 refinements) is presented on figure \ref{fig:2dresult}, which corresponds to the steady state reached by the system.
\begin{figure}[h!]
    \centering
    \includegraphics[scale=0.35]{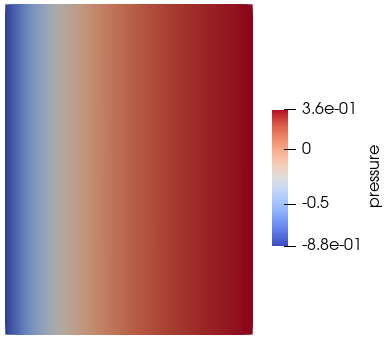}
    \caption{Expected solution for the pressure in the 2D experiment, computed with order 6 and 5 refinements. It is the steady state of the system. Visualization: \cite{paraview}.}
    \label{fig:2dresult}
\end{figure}
\newpage
For illustration of how the system changed, the initial condition for the pressure is shown in figure \ref{fig:2dinitial}.
\begin{figure}[h!]
    \centering
    \includegraphics[scale=0.35]{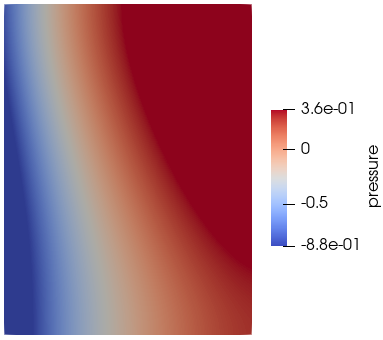}
    \caption{Initial condition for pressure in the 2D experiment, computed with order 6 and 5 refinements. Visualization: \cite{paraview}.}
    \label{fig:2dinitial}
\end{figure}\\
Notice that the initial condition has a higher pressure on the upper-right part of the domain and it's not uniform along any of the axes. However, after the system reaches the steady state pressure is constant for a fixed $x$ value.\enter
Now, the experiment consists on iterating through different orders and changing the refinement level for each of the orders, in order to check differences in the graphical solution. The experiment was done using a time step of $dt=0.001$ with a total time of $T=0.05$. It was computed with the parallel version of MFEM library (with 4 cores), using the Navier Miniapp \cite{MFEM}. The system was solved 36 times, corresponding to $order=1,2,3,4,5,6$, and for each of them, $\#refinements=0,1,2,3,4,5$. After checking all of the results in ParaView \cite{paraview}, we summarized the general behaviour of the steady state solutions, as shown on figure \ref{fig:2dexperiment}. As notation, each of the results has a corresponding $(k,r)$ value, where $k$ denotes the order and $r$ the amount of refinements made.
\begin{figure}[h!]
    \centering
    \includegraphics[scale=0.5]{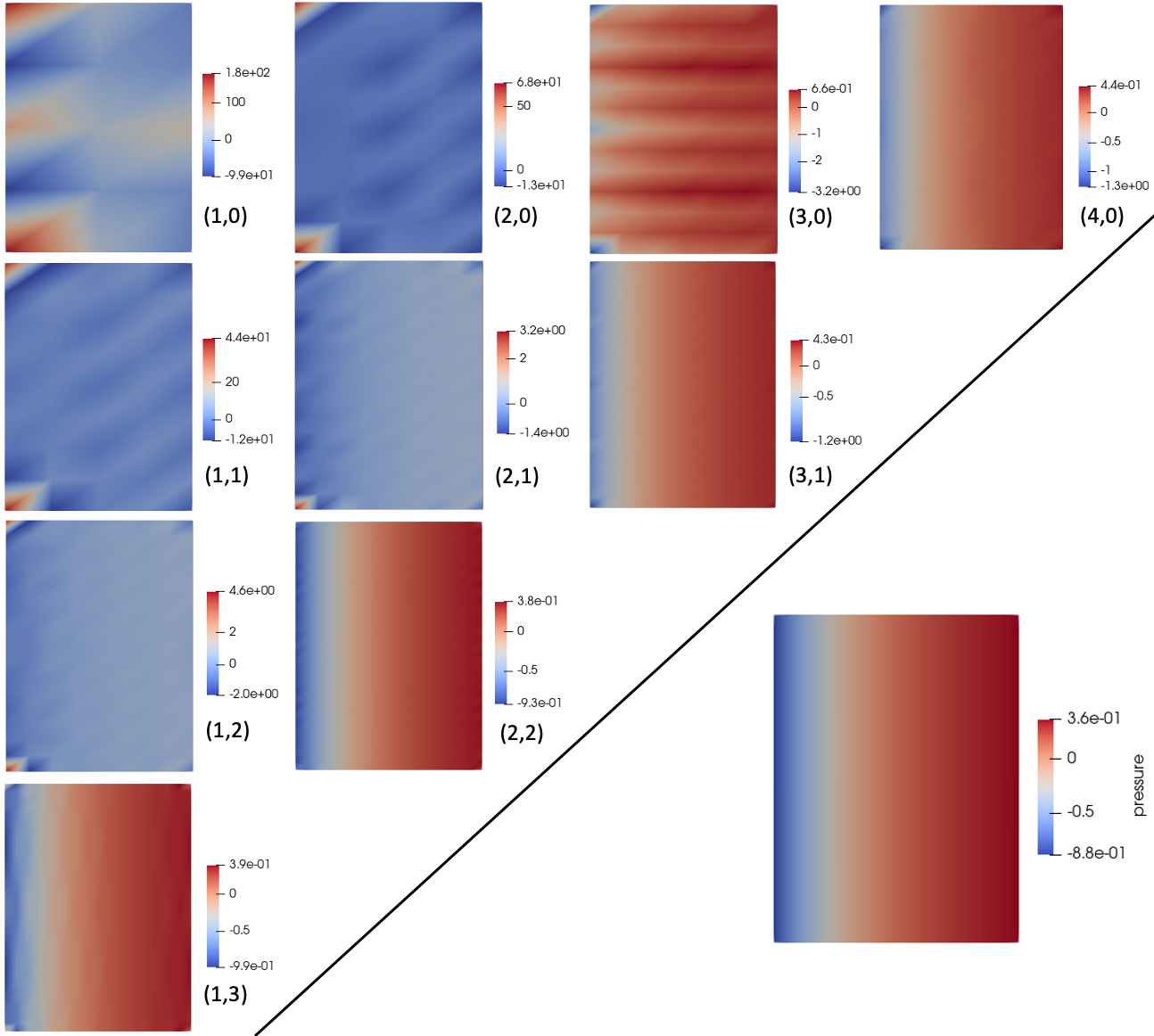}
    \caption{Summary of the steady state solutions for the 2D experiment as order and refinement level changes. The orders and refinement levels that don't appear are almost identical to the expected steady state. Visualization: \cite{paraview}.}
    \label{fig:2dexperiment}
\end{figure}\\
From the summary it is clear that with low order and low refinement levels, the approximation for the solution is not a good one. It's interesting to notice that increasing $1$ in the order can be similar to increasing $1$ in the number of refinements, as seen on the steady state of $(1,1)\sim(2,0)$, $(1,2)\sim(2,1)$ and $(3,1)\sim(4,0)$.\enter
Furthermore, with a high order, the solution converges as expected even if the mesh is not refined. Finally, note that for the lowest order or for the lowest refinement level, an increase of $5$ in the counterpart achieves a good approximation (the steady states that are not shown are the ones that already achieve a good approximation).

\subsection{3D Experiment: Turbulence }\label{sec:NS3D}
The three dimensional domain is a $0.4\times0.4\times2.5$ parallelepiped with a vertex on the origin $(0,0,0)$ and the other ones on the positive part of the coordinate system. It has a cylindrical hole, parallel to the $yx$-plane, whose cross section center is located at $(0.5,0.2)$ and has a radius of $0.05$. The domain is shown on Figure \ref{fig:3Domain}.
\begin{figure}[h!]
    \centering
    \includegraphics[scale=0.5]{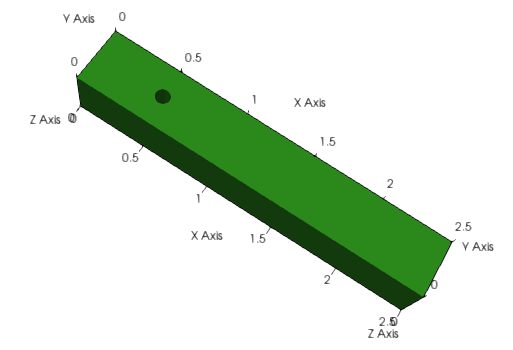}
    \caption{Three dimensional domain used in the 3D Navier-Stokes experiments. It is a $0.4\times0.4\times2.5$ parallelepiped with a cylindrical hole. Visualization: \cite{paraview}.}
    \label{fig:3Domain}
\end{figure}\\
For the experiment, both the initial condition and the boundary condition for velocity were settled to be the same:
\begin{equation}\label{eq:vel3D}
    u_0\begin{pmatrix}
    x\\y\\z\\t
    \end{pmatrix}=u_i\begin{pmatrix}
    x\\y\\z\\t
    \end{pmatrix}=\begin{pmatrix}
    u_x(x,y,z,t)\\0\\0,
    \end{pmatrix}
\end{equation}
where 
\begin{equation}\label{eq:velx3D}
    u_x(x,y,z,t)=\left\{\begin{split}
    \frac{36yz}{0.41^4}\sin\left(\frac{\pi t}{8}\right)(0.41-y)(0.41-z),\text{ if }x\leq 10^{-8},\\
    0,\text{ if }x>10^{-8}.
\end{split}\right.
\end{equation}
Notice that the velocity condition simulates a system where the fluid is entering through the squared face of the domain near the hole. Also, the boundary condition was forced only on the inlet and on the walls of the domain.
\newpage 
Furthermore, the experiment was computed with $4$-th order elements, using a time step of $dt=0.001$, a total time of $T=8$ and the parameter of kinematic viscosity for the fluid being $\nu=0.001$.\enter
For the visualization of the solutions, we used ParaView's \cite{paraview} stream tracer functionality, which shows the stream lines of the system at a given moment. Recall that the stream lines are tangent to the vector field of the velocity, and they show the trajectory of particles through the field (at a given instant of time). Also, there are 5 visualizations, corresponding to times $t=0,2,4,6,8$, and have the value of pressure codified in a color scale.\enter
At $t=0$, the stream lines show a laminar flow (all lines are almost parallel and follow the same direction) that avoids the obstacle, and the pressure is high on the left because the fluid is entering though that part of the domain, as shown on figure \ref{fig:3D1}.
\begin{figure}[h!]
    \centering
    \includegraphics[scale=0.4]{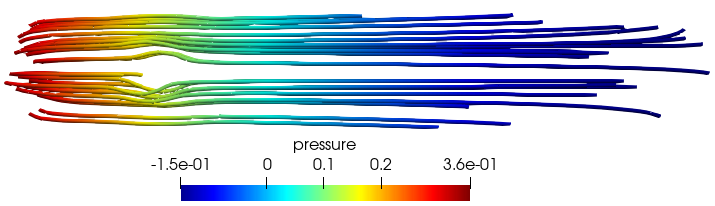}
    \caption{Stream lines for the velocity of the 3D experiment at $t=0$. Visualization: \cite{paraview}.}
    \label{fig:3D1}
\end{figure}\\
Then, at $t=2$, more fluid is coming in, therefore, the pressure on the left increases. However, the value of pressure after the obstacle starts to have some variations, which will cause the turbulence later.
\begin{figure}[h!]
    \centering
    \includegraphics[scale=0.4]{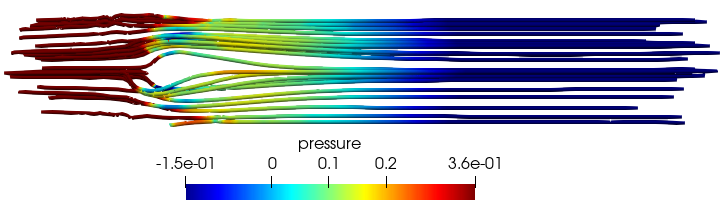}
    \caption{Stream lines for the velocity of the 3D experiment at $t=2$. Visualization: \cite{paraview}.}
    \label{fig:3D2}
\end{figure}\\
At $t=4$, the turbulence starts to show up. As seen on figure \ref{fig:3D3}, after the obstacle, some of the stream lines have spiral forms.
\newpage
\begin{figure}[h!]
    \centering
    \includegraphics[scale=0.4]{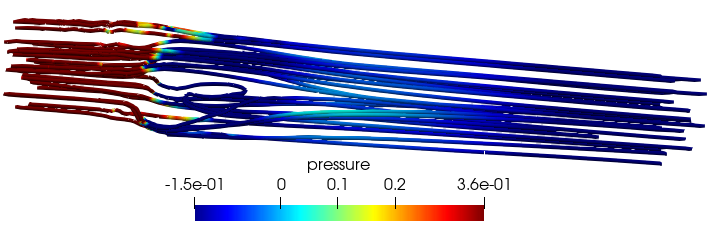}
    \caption{Stream lines for the velocity of the 3D experiment at $t=4$. Visualization: \cite{paraview}.}
    \label{fig:3D3}
\end{figure}
Then, at $t=6$, the pressure on the left finally starts to lower (with a lot of variation) and the pressure on the right starts to increase, because the fluid is already passing through the obstacle and no more fluid is coming in.. More spiral-shaped stream lines show up after the obstacle.
\begin{figure}[h!]
    \centering
    \includegraphics[scale=0.4]{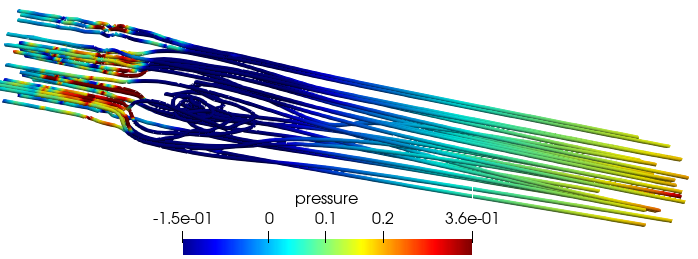}
    \caption{Stream lines for the velocity of the 3D experiment at $t=6$. Visualization: \cite{paraview}.}
    \label{fig:3D4}
\end{figure}\\
Finally, at $t=T=8$, the pressure on the left is low and on the right is high. However, the high variation of pressure in previous time steps generated a lot of turbulence near the obstacle. The fluid presents two states, a laminar one, and a turbulent one. The turbulent state, characterized by spiral movement and swirls, is presented near the obstacle; while the laminar state, characterized by straight lines, is presented away from the obstacle. As seen on figure \ref{fig:3D5}.
\begin{figure}[h!]
    \centering
    \includegraphics[scale=0.4]{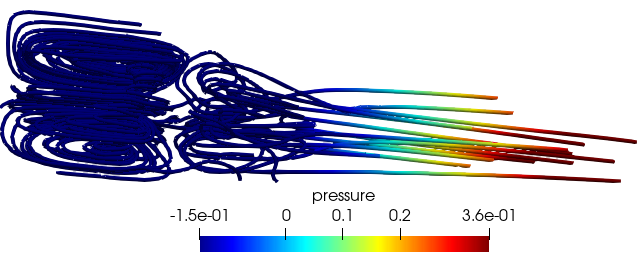}
    \caption{Stream lines for the velocity of the 3D experiment at $t=T=8$. Visualization: \cite{paraview}.}
    \label{fig:3D5}
\end{figure}\\
\newpage

\section{Conclusion and Perspectives}\label{sec:conclusion}
The MFEM library allows us to approximate the solution of partial differential equations in a versatile way. Moreover, the library has a lot of potential because it can compute with high order elements without requiring a very powerful computer, for example, we could run experiments with elements of order 10 (which is a relatively high order). Also, the navier mini app of the library provides a simple, precise and efficient way for simulating dynamical systems that involve incompressible fluids. Furthermore, it is important to use a good visualization tool, preferably one that allows the visualization of vector fields and stream lines when working with fluid equations. Finally, recall that finite element methods enable the study of systems that depend on difficult partial differential equations.\enter
Following this work, some study can be made on some of the following topics:
\begin{itemize}
    \item Fluid modeling via partial differential equations.
    \item Error theorems for finite element methods.
    \item Effects of the mesh in the solution.
    \item Solutions for the equations when the parameters are not constants, but functions.
    \item Discretization of time.
    \item Picking of the time step $dt$, in order to achieve an appropriate solution.
    \item Turbulence and laminar flow.
\end{itemize}

\newpage

\section{Appendices}\label{sec:appendix}
\subsection{Appendix A : Code for comparison}\label{sec:appA}
Here, the code used for Section \ref{sec:Vs} (written in C++) is shown, with a brief explanations of it's functionality.\\
\definecolor{fondo}{rgb}{0.95,0.95,0.95}
\definecolor{comentarios}{rgb}{0,0.35,0.05}
\definecolor{claves}{rgb}{0,0,0.4}
\lstset{language=C++,
backgroundcolor=\color{fondo},
commentstyle=\color{comentarios},
stringstyle=\color{comentarios},
keywordstyle=\bfseries\color{claves},
breaklines=true
}
\xitem Include the required libraries (including MFEM) and begin main function.
\begin{lstlisting}[frame=single]
#include "mfem.hpp"
#include <fstream>
#include <iostream>
using namespace std;
using namespace mfem;
int main(int argc, char *argv[]){
\end{lstlisting}
\xitem Parse command-line options (in this project we only change "order" option) and print them.
\begin{lstlisting}[frame=single]
const char *mesh_file = "../data/star.mesh";
int order = 1;
bool visualization = true;
OptionsParser args(argc, argv);
args.AddOption(&mesh_file, "-m", "--mesh",
               "Mesh file to use.");
args.AddOption(&order, "-o", "--order",
            "Finite element order (polynomial degree).");
args.AddOption(&visualization, "-vis", "--visualization", "-no-vis", "--no-visualization",
               "Enable or disable GLVis visualization.");
args.Parse();
if (!args.Good()){
   args.PrintUsage(cout);
   return 1;
}
args.PrintOptions(cout);
\end{lstlisting}
\xitem Create mesh object from the \texttt{star.mesh} file and get it's dimension.
\begin{lstlisting}[frame=single]
Mesh *mesh = new Mesh(mesh_file,1,1); 
int dim = mesh->Dimension();
\end{lstlisting}
\newpage
\xitem Refine the mesh a given number of times (\texttt{UniformRefinement}).
\begin{lstlisting}[frame=single]
int ref_levels;
cout << "Refinements: ";
cin >> ref_levels;
for (int l = 0; l < ref_levels; l++){
mesh->UniformRefinement();
}
\end{lstlisting}
\xitem Get size indicator for mesh size (\texttt{h\_max}) and print it.
\begin{lstlisting}[frame=single]
double mesh_size, h = 0;
for (int i=0;i<mesh->GetNE();i++){
    mesh_size = mesh->GetElementSize(i,2);
    if(mesh_size>h){
        h = mesh_size;
    }
}
cout << "h: " << h << endl;
\end{lstlisting}
\xitem Define finite element spaces. For mixed finite element method, the order will be one less than for Lagrange finite element method. The last one is a vector $L^2$ space that we will use later to get mixed velocity components.
\begin{lstlisting}[frame=single]
FiniteElementCollection *H1 = new H1_FECollection(order, dim);
FiniteElementSpace *H1_space = new FiniteElementSpace(mesh, H1);
FiniteElementCollection *hdiv_coll(new RT_FECollection(order-1, dim));
FiniteElementCollection *l2_coll(new L2_FECollection(order-1, dim));
FiniteElementSpace *R_space = new FiniteElementSpace(mesh, hdiv_coll);
FiniteElementSpace *W_space = new FiniteElementSpace(mesh, l2_coll);
FiniteElementSpace *V_space = new FiniteElementSpace(mesh, l2_coll, 2);
\end{lstlisting}
\newpage
\xitem Define the parameters of the mixed problem. \texttt{C} functions are defined at the end. Boundary condition is natural.
\begin{lstlisting}[frame=single]
ConstantCoefficient k(1.0);
void fFun(const Vector & x, Vector & f);
VectorFunctionCoefficient fcoeff(dim, fFun);
double gFun(const Vector & x);
FunctionCoefficient gcoeff(gFun);
double f_bound(const Vector & x);
FunctionCoefficient fbndcoeff(f_bound);
\end{lstlisting}
\xitem Define the parameters of the Lagrange problem. Boundary condition is essential.
\begin{lstlisting}[frame=single]
ConstantCoefficient one(1.0);
Array<int> ess_tdof_list;
if (mesh->bdr_attributes.Size()){
   Array<int> ess_bdr(mesh->bdr_attributes.Max());
   ess_bdr = 1;
   H1_space->GetEssentialTrueDofs(ess_bdr, ess_tdof_list);
}
\end{lstlisting}
\xitem Define the exact solution. \texttt{C} functions are defined at the end.
\begin{lstlisting}[frame=single]
void u_ex(const Vector & x, Vector & u);
double p_ex(const Vector & x);
double u_ex_x(const Vector & x);
double u_ex_y(const Vector & x);
\end{lstlisting}
\xitem Get space dimensions and create vectors for the right hand side.
\begin{lstlisting}[frame=single]
Array<int> block_offsets(3);
block_offsets[0] = 0;
block_offsets[1] = R_space->GetVSize();
block_offsets[2] = W_space->GetVSize();
block_offsets.PartialSum();
BlockVector rhs_mixed(block_offsets);
Vector rhs(H1_space->GetVSize());
\end{lstlisting}
\newpage
\xitem Define the right hand side. These are \texttt{LinearForm} objects associated to some finite element space and rhs vector. "f" and "g" are for the mixed method and "b" is for the Lagrange method. Also, "rhs" vectors are the variables that store the information of the right hand side.
\begin{lstlisting}[frame=single]
LinearForm *fform(new LinearForm);
fform->Update(R_space, rhs_mixed.GetBlock(0), 0);
fform->AddDomainIntegrator(new VectorFEDomainLFIntegrator(fcoeff));
fform->AddBoundaryIntegrator(new VectorFEBoundaryFluxLFIntegrator(fbndcoeff));
fform->Assemble();

LinearForm *gform(new LinearForm);
gform->Update(W_space, rhs_mixed.GetBlock(1), 0);
gform->AddDomainIntegrator(new DomainLFIntegrator(gcoeff));
gform->Assemble();

LinearForm *b(new LinearForm);
b->Update(H1_space, rhs, 0);
b->AddDomainIntegrator(new DomainLFIntegrator(one));
b->Assemble();
\end{lstlisting}
\xitem Create variables to store the solution. "x" is the \texttt{Vector} used as input in the iterative method.
\begin{lstlisting}[frame=single]
BlockVector x_mixed(block_offsets);
GridFunction u_mixed(R_space), p_mixed(W_space), ux_mixed(W_space), uy_mixed(W_space), ue(V_space);
Vector x(H1_space->GetVSize());
GridFunction ux(W_space),uy(W_space),p(H1_space);
\end{lstlisting}
\newpage
\xitem Define the left hand side for mixed method. This is the \texttt{BilinearForm} representing the Darcy matrix \eqref{LaplaceMixedEnd}. \texttt{VectorFEMMassIntegrator} is asociated to $k*u-\nabla p$ and \texttt{VectorFEDDivergenceIntegrator} is asociated to $\mbox{div}(u)$.
\begin{lstlisting}[frame=single]
BilinearForm *mVarf(new BilinearForm(R_space));
MixedBilinearForm *bVarf(new MixedBilinearForm(R_space, W_space));
mVarf->AddDomainIntegrator(new VectorFEMassIntegrator(k));
mVarf->Assemble();
mVarf->Finalize();
SparseMatrix &M(mVarf->SpMat());
bVarf->AddDomainIntegrator(new VectorFEDivergenceIntegrator);
bVarf->Assemble();
bVarf->Finalize();
SparseMatrix & B(bVarf->SpMat());
B *= -1.;
SparseMatrix *BT = Transpose(B);
BlockMatrix D(block_offsets);
D.SetBlock(0,0, &M);
D.SetBlock(0,1, BT);
D.SetBlock(1,0, &B);
\end{lstlisting}
\xitem Define the left hand side for Lagrange method. This is the \texttt{BilinearForm} asociated to the laplacian operator. \texttt{DiffusionIntegrator} is asociated to $\Delta u$. The method \texttt{FormLinearSystem} is only used to establish the essential boundary condition.
\begin{lstlisting}[frame=single]
OperatorPtr A;
Vector XX,BB;
BilinearForm *a(new BilinearForm(H1_space));
a->AddDomainIntegrator(new DiffusionIntegrator(one));
a->Assemble();
a->FormLinearSystem(ess_tdof_list, p, *b, A, XX, BB);
\end{lstlisting}
\newpage
\xitem Solve linear systems with MINRES (for mixed) and CG (for Lagrange). \texttt{SetOperator} method establishes the lhs. \texttt{Mult} method executes the iterative algorithm and receives as input: the rhs and the vector to store the solution.
\begin{lstlisting}[frame=single]    
int maxIter(10000); 
double rtol(1.e-6);
double atol(1.e-10);
    
MINRESSolver Msolver;
Msolver.SetAbsTol(atol);
Msolver.SetRelTol(rtol);
Msolver.SetMaxIter(maxIter);
Msolver.SetPrintLevel(0); 
Msolver.SetOperator(D); 
x_mixed = 0.0;
Msolver.Mult(rhs_mixed, x_mixed);
if (Msolver.GetConverged())
   std::cout << "MINRES converged in " << Msolver.GetNumIterations() << " iterations with a residual norm of " << Msolver.GetFinalNorm() << ".\n";
else
   std::cout << "MINRES did not converge in " << Msolver.GetNumIterations() << " iterations. Residual norm is " << Msolver.GetFinalNorm() << ".\n";
   
CGSolver Lsolver;
Lsolver.SetAbsTol(atol);
Lsolver.SetRelTol(rtol);
Lsolver.SetMaxIter(maxIter);
Lsolver.SetPrintLevel(0);
Lsolver.SetOperator(*A);
x = 0.0;
Lsolver.Mult(rhs,x);
if (Lsolver.GetConverged())
   std::cout << "CG converged in " << Lsolver.GetNumIterations() << " iterations with a residual norm of " << Lsolver.GetFinalNorm() << ".\n";
else
   std::cout << "CG did not converge in " << Lsolver.GetNumIterations() << " iterations. Residual norm is " << Lsolver.GetFinalNorm() << ".\n";
\end{lstlisting}
\newpage
\xitem Save the solution into \texttt{GridFunctions}, which are used for error computation and visualization.
\begin{lstlisting}[frame=single]  
u_mixed.MakeRef(R_space, x_mixed.GetBlock(0), 0);
p_mixed.MakeRef(W_space, x_mixed.GetBlock(1), 0);
p.MakeRef(H1_space,x,0);
\end{lstlisting}
\xitem Get missing velocities from the solutions obtained.\\Remember that $u = -\nabla p$. Mixed components are extracted using the auxiliary variable "\texttt{ue}" defined before.
\begin{lstlisting}[frame=single]  
p.GetDerivative(1,0,ux); 
p.GetDerivative(1,1,uy); 
ux *= -1;
uy *= -1;

VectorGridFunctionCoefficient uc(&u_mixed);
ue.ProjectCoefficient(uc);
GridFunctionCoefficient ux_mixed_coeff(&ue,1);
GridFunctionCoefficient uy_mixed_coeff(&ue,2);
ux_mixed.ProjectCoefficient(ux_mixed_coeff);
uy_mixed.ProjectCoefficient(uy_mixed_coeff);
\end{lstlisting}
\xitem Create the asociated \texttt{Coefficient} objects for error computation.
\begin{lstlisting}[frame=single]  
GridFunction* pp = &p;
GridFunctionCoefficient p_coeff(pp);
GridFunction* uxp = &ux;
GridFunction* uyp = &uy;
GridFunctionCoefficient ux_coeff(uxp);
GridFunctionCoefficient uy_coeff(uyp);
FunctionCoefficient pex_coeff(p_ex); 
VectorFunctionCoefficient uex_coeff(dim,u_ex); 
FunctionCoefficient uex_x_coeff(u_ex_x);
FunctionCoefficient uex_y_coeff(u_ex_y);
\end{lstlisting}
\xitem Define integration rule.
\begin{lstlisting}[frame=single]  
int order_quad = max(2, 2*order+1);
const IntegrationRule *irs[Geometry::NumGeom];
for (int i=0; i < Geometry::NumGeom; ++i){
    irs[i] = &(IntRules.Get(i, order_quad));
}
\end{lstlisting}
\newpage
\xitem Compute exact solution norms. Here, the parameter $2$ in \texttt{ComputeLpNorm} makes reference to the $L^2$ norm.
\begin{lstlisting}[frame=single]  
double norm_p = ComputeLpNorm(2., pex_coeff, *mesh, irs);
double norm_u = ComputeLpNorm(2., uex_coeff, *mesh, irs); 
double norm_ux = ComputeLpNorm(2., uex_x_coeff, *mesh, irs); 
double norm_uy = ComputeLpNorm(2., uex_y_coeff, *mesh, irs); 
\end{lstlisting}
\xitem Compute and print absolute errors.
\begin{lstlisting}[frame=single]  
double abs_err_u_mixed = u_mixed.ComputeL2Error(uex_coeff,irs);
printf("Velocity Mixed Absolute Error: %e\n", abs_err_u_mixed / norm_u);
double abs_err_p_mixed = p_mixed.ComputeL2Error(pex_coeff,irs);
printf("Pressure Mixed Absolute Error: %e\n", abs_err_p_mixed / norm_p);
double abs_err_p = p.ComputeL2Error(pex_coeff,irs); 
printf("Pressure Absolute Error: %e\n", abs_err_p / norm_p);
double abs_err_ux = ux.ComputeL2Error(uex_x_coeff,irs);
double abs_err_uy = uy.ComputeL2Error(uex_y_coeff,irs);
double abs_err_u = pow(pow(abs_err_ux,2)+pow(abs_err_uy,2),0.5);
printf("Velocity Absolute Error: %e\n", abs_err_u / norm_u);
\end{lstlisting}
\xitem Compute and print comparison errors.
\begin{lstlisting}[frame=single]  
double err_ux = ux_mixed.ComputeL2Error(ux_coeff,irs);
double err_uy = uy_mixed.ComputeL2Error(uy_coeff,irs);
double err_u = pow(pow(err_ux,2)+pow(err_uy,2),0.5);
printf("Velocity Comparison Error: %e\n", err_u / norm_u);
double err_p = p_mixed.ComputeL2Error(p_coeff, irs);
printf("Pressure Comparison Error: %e\n", err_p / norm_p);
\end{lstlisting}
\newpage
\xitem Visualize the solutions and the domain. \texttt{GLVis} visualization tool uses port $19916$ to receive data.
\begin{lstlisting}[frame=single]  
char vishost[] = "localhost";
int  visport   = 19916;
if(visualization){
    Vector x_domain(H1_space->GetVSize());
    GridFunction domain(H1_space);
    x_domain=0.0;
    domain.MakeRef(H1_space,x_domain,0);
    socketstream dom_sock(vishost, visport);
    dom_sock.precision(8);
    dom_sock << "solution\n" << *mesh << domain << "window_title 'Domain'" << endl;
    
    socketstream um_sock(vishost, visport);
    um_sock.precision(8);
    um_sock << "solution\n" << *mesh << u_mixed << "window_title 'Velocity Mixed'" << endl;
    socketstream pm_sock(vishost, visport);
    pm_sock.precision(8);
    pm_sock << "solution\n" << *mesh << p_mixed << "window_title 'Pressure Mixed'" << endl;
    socketstream uxm_sock(vishost, visport);
    uxm_sock.precision(8);
    uxm_sock << "solution\n" << *mesh << ux_mixed << "window_title 'X Velocity Mixed'" << endl;
    socketstream uym_sock(vishost, visport);
    uym_sock.precision(8);
    uym_sock << "solution\n" << *mesh << uy_mixed << "window_title 'Y Velocity Mixed'" << endl;
    
    socketstream p_sock(vishost, visport);
    p_sock.precision(8);
    p_sock << "solution\n" << *mesh << p << "window_title 'Pressure'" << endl;
    socketstream ux_sock(vishost, visport);
    ux_sock.precision(8);
    ux_sock << "solution\n" << *mesh << ux << "window_title 'X Velocity'" << endl;
    socketstream uy_sock(vishost, visport);
    uy_sock.precision(8);
    uy_sock << "solution\n" << *mesh << uy << "window_title 'Y Velocity'" << endl;
}}
\end{lstlisting}
\xitem Define \texttt{C} functions.
\begin{lstlisting}[frame=single]  
void fFun(const Vector & x, Vector & f){
    f = 0.0;
}
double gFun(const Vector & x){
    return -1.0;
}
double f_bound(const Vector & x){
    return 0.0;
}
void u_ex(const Vector & x, Vector & u){
   double xi(x(0));
   double yi(x(1));
   double zi(0.0);
   u(0) = - exp(xi)*sin(yi)*cos(zi);
   u(1) = - exp(xi)*cos(yi)*cos(zi);
}
double u_ex_x(const Vector & x){
   double xi(x(0));
   double yi(x(1));
   double zi(0.0);
   return -exp(xi)*sin(yi)*cos(zi);
}
double u_ex_y(const Vector & x){
   double xi(x(0));
   double yi(x(1));
   double zi(0.0);
   return -exp(xi)*cos(yi)*cos(zi);
}
double p_ex(const Vector & x){
   double xi(x(0));
   double yi(x(1));
   double zi(0.0);
   return exp(xi)*sin(yi)*cos(zi);
}
\end{lstlisting}

\newpage
\subsection{Appendix B : Numerical values of the comparison}\label{sec:appB}
The \textit{order} parameter will be fixed for each table and $h$ parameter is shown in the first column. To interpret the results take into account that \textbf{P} refers to pressure, \textbf{U} refers to velocity, \textbf{mx} refers to mixed (from mixed finite element method), \textbf{err} refers to absolute error (compared to the exact solution), and \textbf{comp} refers to comparison (the error between the two solutions obtained by the two different methods).\\
\newline
\indent
\textbf{\textit{Order = 1}}
\begin{table}[h!]
\scalebox{0.7}{
\begin{tabular}{|c|c|c|c|c|c|c|}
\hline
\textbf{h} & \textbf{P comp} & \textbf{P err} & \textbf{Pmx err} & \textbf{U comp} & \textbf{U err} & \textbf{U mx err} \\ \hline
0.572063   & 7.549479e-02    & 1.021287e+00   & 1.025477e+00     & 3.680827e-02    & 1.029378e+00   & 1.037635e+00      \\ \hline
0.286032   & 3.627089e-02    & 1.022781e+00   & 1.023990e+00     & 1.727281e-02    & 1.032760e+00   & 1.035055e+00      \\ \hline
0.143016   & 1.791509e-02    & 1.023236e+00   & 1.023596e+00     & 9.222996e-03    & 1.033725e+00   & 1.034369e+00      \\ \hline
0.0715079  & 8.922939e-03    & 1.023372e+00   & 1.023480e+00     & 5.111295e-03    & 1.033999e+00   & 1.034182e+00      \\ \hline
0.035754   & 4.455715e-03    & 1.023412e+00   & 1.023445e+00     & 2.859769e-03    & 1.034077e+00   & 1.034130e+00      \\ \hline
0.017877   & 2.226845e-03    & 1.023424e+00   & 1.023435e+00     & 1.603788e-03    & 1.034100e+00   & 1.034115e+00      \\ \hline
\end{tabular}}
\end{table}

\textbf{\textit{Order = 2}}
\begin{table}[h!]
\scalebox{0.7}{
\begin{tabular}{|c|c|c|c|c|c|c|}
\hline
\textbf{h} & \textbf{P comp} & \textbf{P err} & \textbf{Pmx err} & \textbf{U comp} & \textbf{U err} & \textbf{U mx err} \\ \hline
0.572063   & 8.069013e-03    & 1.023329e+00   & 1.023554e+00     & 1.399079e-02    & 1.033924e+00   & 1.034255e+00      \\ \hline
0.286032   & 2.138257e-03    & 1.023391e+00   & 1.023470e+00     & 7.845012e-03    & 1.034056e+00   & 1.034146e+00      \\ \hline
0.143016   & 5.704347e-04    & 1.023417e+00   & 1.023442e+00     & 4.400448e-03    & 1.034093e+00   & 1.034120e+00      \\ \hline
0.0715079  & 1.537926e-04    & 1.023426e+00   & 1.023434e+00     & 2.469526e-03    & 1.034104e+00   & 1.034112e+00      \\ \hline
0.035754   & 4.194302e-05    & 1.023428e+00   & 1.023431e+00     & 1.385966e-03    & 1.034107e+00   & 1.034110e+00      \\ \hline
\end{tabular}}
\end{table}

\textbf{\textit{Order = 3}}
\begin{table}[h!]
\scalebox{0.7}{
\begin{tabular}{|c|c|c|c|c|c|c|}
\hline
\textbf{h} & \textbf{P comp} & \textbf{P err} & \textbf{Pmx err} & \textbf{U comp} & \textbf{U err} & \textbf{U mx err} \\ \hline
0.572063   & 8.691241e-04    & 1.023389e+00   & 1.023471e+00     & 8.745151e-03    & 1.034060e+00   & 1.034143e+00      \\ \hline
0.286032   & 2.477673e-04    & 1.023417e+00   & 1.023443e+00     & 4.911967e-03    & 1.034094e+00   & 1.034120e+00      \\ \hline
0.143016   & 7.316263e-05    & 1.023426e+00   & 1.023434e+00     & 2.756849e-03    & 1.034104e+00   & 1.034112e+00      \\ \hline
0.0715079  & 2.178864e-05    & 1.023428e+00   & 1.023431e+00     & 1.547232e-03    & 1.034108e+00   & 1.034110e+00      \\ \hline
\end{tabular}}
\end{table}

\textbf{\textit{Order = 4}}
\begin{table}[h!]
\scalebox{0.7}{
\begin{tabular}{|c|c|c|c|c|c|c|}
\hline
\textbf{h} & \textbf{P comp} & \textbf{P err} & \textbf{Pmx err} & \textbf{U comp} & \textbf{U err} & \textbf{U mx err} \\ \hline
0.572063   & 3.199774e-04    & 1.023412e+00   & 1.023448e+00     & 6.119857e-03    & 1.034088e+00   & 1.034124e+00      \\ \hline
0.286032   & 9.547574e-05    & 1.023424e+00   & 1.023435e+00     & 3.434952e-03    & 1.034103e+00   & 1.034114e+00      \\ \hline
0.143016   & 2.862666e-05    & 1.023428e+00   & 1.023431e+00     & 1.927814e-03    & 1.034107e+00   & 1.034111e+00      \\ \hline
\end{tabular}}
\end{table}

\newpage
\textbf{\textit{Order = 5}}
\begin{table}[h!]
\scalebox{0.7}{
\begin{tabular}{|c|c|c|c|c|c|c|}
\hline
\textbf{h} & \textbf{P comp} & \textbf{P err} & \textbf{Pmx err} & \textbf{U comp} & \textbf{U err} & \textbf{U mx err} \\ \hline
0.572063   & 1.552006e-04    & 1.023420e+00   & 1.023439e+00     & 4.578518e-03    & 1.034099e+00   & 1.034117e+00      \\ \hline
0.286032   & 4.658038e-05    & 1.023427e+00   & 1.023433e+00     & 2.569749e-03    & 1.034106e+00   & 1.034112e+00      \\ \hline
0.143016   & 1.406993e-05    & 1.023429e+00   & 1.023431e+00     & 1.442205e-03    & 1.034108e+00   & 1.034110e+00      \\ \hline
\end{tabular}}
\end{table}

\textbf{\textit{Order = 6}}
\begin{table}[h!]
\scalebox{0.7}{
\begin{tabular}{|c|c|c|c|c|c|c|}
\hline
\textbf{h} & \textbf{P comp} & \textbf{P err} & \textbf{Pmx err} & \textbf{U comp} & \textbf{U err} & \textbf{U mx err} \\ \hline
0.572063   & 8.612580e-05    & 1.023424e+00   & 1.023435e+00     & 3.584133e-03    & 1.034103e+00   & 1.034114e+00      \\ \hline
0.286032   & 2.600417e-05    & 1.023428e+00   & 1.023431e+00     & 2.011608e-03    & 1.034107e+00   & 1.034111e+00      \\ \hline
0.143016   & 7.897631e-06    & 1.023429e+00   & 1.023430e+00     & 1.128989e-03    & 1.034109e+00   & 1.034110e+00      \\ \hline
\end{tabular}}
\end{table}

\textbf{\textit{Order = 7}}
\begin{table}[h!]
\scalebox{0.7}{
\begin{tabular}{|c|c|c|c|c|c|c|}
\hline
\textbf{h} & \textbf{P comp} & \textbf{P err} & \textbf{Pmx err} & \textbf{U comp} & \textbf{U err} & \textbf{U mx err} \\ \hline
0.572063   & 5.243187e-05    & 1.023426e+00   & 1.023433e+00     & 2.899307e-03    & 1.034105e+00   & 1.034112e+00      \\ \hline
0.286032   & 1.589631e-05    & 1.023429e+00   & 1.023431e+00     & 1.627221e-03    & 1.034108e+00   & 1.034110e+00      \\ \hline
\end{tabular}}
\end{table}

\textbf{\textit{Order = 8}}
\begin{table}[h!]
\scalebox{0.7}{
\begin{tabular}{|c|c|c|c|c|c|c|}
\hline
\textbf{h} & \textbf{P comp} & \textbf{P err} & \textbf{Pmx err} & \textbf{U comp} & \textbf{U err} & \textbf{U mx err} \\ \hline
0.572063   & 3.409225e-05    & 1.023427e+00   & 1.023432e+00     & 2.404311e-03    & 1.034107e+00   & 1.034111e+00      \\ \hline
0.286032   & 1.037969e-05    & 1.023429e+00   & 1.023430e+00     & 1.349427e-03    & 1.034108e+00   & 1.034110e+00      \\ \hline
\end{tabular}}
\end{table}

\textbf{\textit{Order = 9}}
\begin{table}[h!]
\scalebox{0.7}{
\begin{tabular}{|c|c|c|c|c|c|c|}
\hline
\textbf{h} & \textbf{P comp} & \textbf{P err} & \textbf{Pmx err} & \textbf{U comp} & \textbf{U err} & \textbf{U mx err} \\ \hline
0.572063   & 2.328387e-05    & 1.023428e+00   & 1.023431e+00     & 2.033288e-03    & 1.034107e+00   & 1.034110e+00      \\ \hline
0.286032   & 7.124397e-06    & 1.023429e+00   & 1.023430e+00     & 1.141177e-03    & 1.034109e+00   & 1.034110e+00      \\ \hline
\end{tabular}}
\end{table}

\textbf{\textit{Order = 10}}
\begin{table}[h!]
\scalebox{0.7}{
\begin{tabular}{|c|c|c|c|c|c|c|}
\hline
\textbf{h} & \textbf{P comp} & \textbf{P err} & \textbf{Pmx err} & \textbf{U comp} & \textbf{U err} & \textbf{U mx err} \\ \hline
0.572063   & 1.664200e-05    & 1.023429e+00   & 1.023431e+00     & 1.746755e-03    & 1.034108e+00   & 1.034110e+00      \\ \hline
0.286032   & 5.085321e-06    & 1.023429e+00   & 1.023430e+00     & 9.803705e-04    & 1.034109e+00   & 1.034109e+00      \\ \hline
\end{tabular}}
\end{table}
\newpage
\subsection{Appendix C : MiniApp Code for Navier-Stokes}\label{sec:appC}
Here, the code used for Section \ref{sec:NSExp} (written in C++) is shown, with a brief explanations of it's functionality.\\
\xitem Include the required libraries (including navier miniapp).
\begin{lstlisting}[frame=single]
#include "navier_solver.hpp"
#include <fstream>
using namespace mfem;
using namespace navier;
using namespace std;
\end{lstlisting}
\xitem Define the context for the problem. In this case, we define parameters for the 2D and 3D experiment. In general, the parameters to define should be $\nu$, $dt$, $T$, $k$ and $\#refinements$.
\begin{lstlisting}[frame=single]
struct NavierContext{
    //Parameters for 2D experiment
    int max_order_steady = 6;
    int max_refinements_steady = 5;
    double kinvis_steady = 1.0 / 40.0;
    double t_final_steady = 50 * 0.001;
    double dt_steady = 0.001;
    double delta = 0.001;
  
    //Parameters for 3D experiment
    double kinvis_3d = 0.001;
    double t_final_3d = 8.0;
    double dt_3d = 1e-3;
} ctx;
\end{lstlisting}
\xitem Define a velocity as a C function. This function represents the initial and boundary conditions for the velocity in the 3D experiment.
\begin{lstlisting}[frame=single]
void vel_3d(const Vector &x, double t, Vector &u){
   double xi = x(0);
   double yi = x(1);
   double zi = x(2);
   double U = 2.25;
   if(xi <= 1e-8){
      u(0) = 16.0 * U * yi * zi * sin(M_PI * t / 8.0) * (0.41 - yi) * (0.41 - zi) / pow(0.41, 4.0);
   }else{ u(0) = 0.0; }
   u(1) = 0.0;
   u(2) = 0.0;}
\end{lstlisting}
\xitem Define another velocity as a C function. This function represents the boundary condition, $u_0$, for the 2D experiment.
\begin{lstlisting}[frame=single]
void vel_steady(const Vector &x, double t, Vector &u){
    double reynolds = 1.0 / ctx.kinvis_steady;
    double lam = 0.5 * reynolds - sqrt(0.25 * reynolds * reynolds + 4.0 * M_PI * M_PI);
    double xi = x(0);
    double yi = x(1);
    u(0) = 1.0 - exp(lam * xi) * cos(2.0 * M_PI * yi);
    u(1) = lam / (2.0 * M_PI) * exp(lam * xi) 
           * sin(2.0 * M_PI * yi);}
\end{lstlisting}
\xitem Define a third velocity as a C function. This function represents the initial condition, $u_i$, for the 2D experiment.
\begin{lstlisting}[frame=single]
void vel(const Vector &x, double t, Vector &u){
    double reynolds = 1.0 / ctx.kinvis_steady;
    double lam = 0.5 * reynolds - sqrt(0.25 * reynolds * reynolds + 4.0 * M_PI * M_PI);
    double xi = x(0);
    double yi = x(1);
    double delta = ctx.delta;
    u(0) = 1.0 - exp(lam * xi) * cos(2.0 * M_PI * yi);
    u(1) = lam / (2.0 * M_PI) * exp(lam * xi) * sin(2.0 * M_PI * yi);
    u(0) = u(0)+delta*(xi+0.5)*(xi-1)*(yi+0.5)*(yi-1.5);
    u(1) = u(1)+delta*(xi+0.5)*(xi-1)*(yi+0.5)*(yi-1.5);
}
\end{lstlisting}
\xitem Begin a function that receives an order and the amount of refinements. This function will compute the 2D experiment. First, it defines the mesh shown on figure \ref{fig:2Dmesh}.
\begin{lstlisting}[frame=single]
void NS_steady(int order, int refinement){
    Mesh *mesh = new Mesh(2, 4, Element::QUADRILATERAL, false, 1.5, 2.0);
    mesh->EnsureNodes();
    GridFunction *nodes = mesh->GetNodes();
    *nodes -= 0.5;
\end{lstlisting}
\xitem Refine the mesh and create the parallel version of the mesh.
\begin{lstlisting}[frame=single]
    for (int i = 0; i < refinement; ++i){
        mesh->UniformRefinement();}
    auto *pmesh = new ParMesh(MPI_COMM_WORLD, *mesh);
    delete mesh;
\end{lstlisting}
\xitem Create the NavierSolver object (receives the mesh, $k$, and $\nu$ as parameters). Set the initial condition, the function \texttt{vel}. And set the boundary condition, the function \texttt{vel\_steady}.
\begin{lstlisting}[frame=single]
    //Create the flow solver
    NavierSolver flowsolver(pmesh, order, ctx.kinvis_steady);
    flowsolver.EnablePA(true);

    //Set the initial condition
    ParGridFunction *u_ic = flowsolver.GetCurrentVelocity();
    VectorFunctionCoefficient u_excoeff(pmesh->
    Dimension(), vel);
    u_ic->ProjectCoefficient(u_excoeff);

    //Add Dirichlet boundary conditions to velocity space
    Array<int> attr(pmesh->bdr_attributes.Max());
    attr = 1;
    flowsolver.AddVelDirichletBC(vel_steady, attr);
\end{lstlisting}
\xitem Set up the problem ($dt$) and define the ParGridFunctions to store the solution.
\begin{lstlisting}[frame=single]
    double t = 0.0;
    bool last_step = false;
    flowsolver.Setup(ctx.dt_steady);
    ParGridFunction *u_gf = flowsolver.GetCurrentVelocity();
    ParGridFunction *p_gf = flowsolver.GetCurrentPressure();
\end{lstlisting}
\xitem Create the ParaView file and associate the variables to save the solution.
\begin{lstlisting}[frame=single]
    ParaViewDataCollection pvdc("STEADY"+to_string(order)+to_string(refinement),pmesh);
    pvdc.SetDataFormat(VTKFormat::BINARY32);
    pvdc.SetHighOrderOutput(true);
    pvdc.SetLevelsOfDetail(order);
    pvdc.SetCycle(0);
    pvdc.SetTime(t);
    pvdc.RegisterField("velocity",u_gf);
    pvdc.RegisterField("pressure",p_gf);
    pvdc.Save();
\end{lstlisting}
\newpage
\xitem Iterate from $t=0$ to $t=T$. For each iteration, the step is taken and saved in the ParaView file.
\begin{lstlisting}[frame=single]
    for(int step = 0; !last_step; ++step){
        //Check for final step
        if (t + ctx.dt_steady >= ctx.t_final_steady - ctx.dt_steady / 2){last_step = true;}

        //Do the step
        flowsolver.Step(t, ctx.dt_steady, step);

        //Save paraview information
        pvdc.SetCycle(step);
        pvdc.SetTime(t);
        pvdc.Save();}
    delete pmesh;}
\end{lstlisting}
\xitem Begin a function that receives an order and the amount of refinements. This function will compute the 3D experiment. First, it defines the mesh associated to the domain shown on figure \ref{fig:3Domain}. This part of the code was planned to be used with different orders and refinement levels, however, it was only run with one case.
\begin{lstlisting}[frame=single]
void NS_3D(int order, int refinement){
    Mesh *mesh = new Mesh("box-cylinder.mesh");
\end{lstlisting}
\xitem Refine the mesh and create the parallel version of the mesh.
\begin{lstlisting}[frame=single]
    for (int i = 0; i < refinement; ++i){
        mesh->UniformRefinement();}
    auto *pmesh = new ParMesh(MPI_COMM_WORLD, *mesh);
    delete mesh;
\end{lstlisting}
\newpage
\xitem Create the NavierSolver object. Set the initial condition, the function \texttt{vel\_3d}. And set the boundary condition, the function \texttt{vel\_3d} too. Notice that the boundary condition is only applied to some parts of the mesh.
\begin{lstlisting}[frame=single]
    //Create the flow solver
    NavierSolver flowsolver(pmesh, order, ctx.kinvis_3d);
    flowsolver.EnablePA(true);

    //Set the initial condition
    ParGridFunction *u_ic = flowsolver.GetCurrentVelocity();
    VectorFunctionCoefficient u_excoeff(pmesh->
    Dimension(), vel_3d);
    u_ic->ProjectCoefficient(u_excoeff);

    //Add Dirichlet boundary conditions to velocity space restricted to selected attributes on the mesh
    Array<int> attr(pmesh->bdr_attributes.Max());
    attr[0] = 1; //Inlet
    attr[2] = 1; //Walls
    flowsolver.AddVelDirichletBC(vel_3d, attr);
\end{lstlisting}
\xitem Set up the problem ($dt$) and define the ParGridFunctions to store the solution.
\begin{lstlisting}[frame=single]
    double t = 0.0;
    bool last_step = false;
    flowsolver.Setup(ctx.dt_3d);
    ParGridFunction *u_gf = flowsolver.GetCurrentVelocity();
    ParGridFunction *p_gf = flowsolver.GetCurrentPressure();
\end{lstlisting}
\newpage
\xitem Create the ParaView file and associate the variables to save the solution.
\begin{lstlisting}[frame=single]
    ParaViewDataCollection pvdc("3D"+to_string(order)+to_string(refinement), pmesh);
    pvdc.SetDataFormat(VTKFormat::BINARY32);
    pvdc.SetHighOrderOutput(true);
    pvdc.SetLevelsOfDetail(order);
    pvdc.SetCycle(0);
    pvdc.SetTime(t);
    pvdc.RegisterField("velocity", u_gf);
    pvdc.RegisterField("pressure", p_gf);
    pvdc.Save();
\end{lstlisting}
\xitem Iterate from $t=0$ to $t=T$. For each iteration, the step is taken and saved in the ParaView file.
\begin{lstlisting}[frame=single]
    for(int step = 0; !last_step; ++step){
        //Check for final step
        if(t + ctx.dt_3d >= ctx.t_final_3d - ctx.dt_3d / 2){last_step = true;}

        //Do the step
        flowsolver.Step(t, ctx.dt_3d, step);

        //Save paraview information every 10 steps
        if (step % 10 == 0){
            pvdc.SetCycle(step);
            pvdc.SetTime(t);
            pvdc.Save();}}
    delete pmesh;}
\end{lstlisting}
\newpage
\xitem Finally, define the main function with a MPI session (for parallel computation). The 2D experiment corresponds to the use of the function \texttt{NS\_steady} while iterating through orders and refinement levels. And, the 3D experiment is running the function \texttt{NS\_3D} with order $4$. No refinements were done because the original mesh (done by MFEM \cite{MFEM}) gives a good solution already.
\begin{lstlisting}[frame=single]
int main(int argc, char *argv[]){

    //Define mpi sessions (parallel programming)
    MPI_Session mpi(argc,argv);

    //Steady experiment
    for(int order=1;order<=ctx.max_order_steady;order++){
        for(int refinements=0;refinements<=ctx.max_refinements_steady;refinements++){
            NS_steady(order,refinements);
        }
    }

    //3D experiment
    NS_3D(4,0);
}
\end{lstlisting}
\textit{Note:} This code was run using the command \texttt{mpirun -n 4}, which uses 4 cores of the computer.
\end{document}